\documentclass[
	sts
]{imsart}

\usepackage{xr}
\startlocaldefs
\usepackage{amsmath,amssymb,amsthm}
\usepackage{graphicx,color}
\usepackage{dsfont}
\usepackage{booktabs}
\usepackage[sort,numbers]{natbib}
\usepackage[normalem]{ulem}
\usepackage[boxed]{algorithm2e}
\usepackage{hyperref}
\usepackage[capitalize]{cleveref}
\usepackage[notcite,notref,final]{showkeys}
\newtheorem{theorem}{Theorem}

\newtheorem{proposition}{Proposition}
\newtheorem{lemma}{Lemma}
\theoremstyle{definition}
\newtheorem{defin}{Definition}

\newcommand{\cD}{\mathcal{D}}

\newcommand{\cG}{\mathcal{G}}

\newcommand{\cN}{\mathcal{N}}
\newcommand{\cO}{\mathcal{O}}
\newcommand{\cP}{\mathcal{P}}
\newcommand{\cQ}{\mathcal{Q}}

\newcommand{\cS}{\mathcal{S}}

\newcommand{\cU}{\mathcal{U}}
\newcommand{\cV}{\mathcal{V}}

\newcommand{\cX}{\mathcal{X}}

\newcommand{\cZ}{\mathcal{Z}}

\newcommand{\RR}{\mathbb{R}}
\newcommand{\R}{\mathbb{R}}

\newcommand{\1}{\mathds{1}}

\newcommand*{\op}[1]{\left\|#1\right\|_{\mathrm{op}}}
\newcommand*{\frob}[1]{\left\|#1\right\|_{\mathrm{F}}}

\newcommand*{\tv}[2]{d_{\mathrm{TV}}(#1, #2)}
\newcommand*{\chis}[2]{\chi^2(#1, #2)}

\newcommand*{\E}{\mathbb E}
\newcommand*{\p}{\mathbb P}
\newcommand*{\ep}{\varepsilon}
\newcommand*{\defeq}{:=}
\newcommand*{\dd}{\, \mathrm{d}}

\DeclareMathOperator*{\argmax}{argmax}

\definecolor{MIT}{cmyk}{.24, 1.00, .78, .17}

\DeclareMathOperator{\var}{Var}

\newcommand*{\tone}[1]{\transp{1}{#1}}

\newcommand*{\tp}[1]{\transp{p}{#1}}
\newcommand*{\tpp}[1]{\transp{p'}{#1}}
\newcommand*{\transp}[2]{T_{#1}(#2)}
\newcommand*{\rate}[1]{r_{p, k}(#1)}
\newcommand*{\rated}[1]{r_{p, d}(#1)}

\newcommand*{\ratedlb}[1]{r'_{p, d}(#1)}
\newcommand*{\lip}{\mathrm{Lip}}
\newcommand*{\wpk}{\tilde W_{p,k}}
\newcommand*{\stie}{\cV_k(\RR^d)}
\newcommand*{\mui}{\mu^{(i)}}
\newcommand*{\muone}{\mu^{(1)}}
\newcommand*{\mutwo}{\mu^{(2)}}
\newcommand*{\nuone}{\nu^{(1)}}
\newcommand*{\nutwo}{\nu^{(2)}}
\newcommand*{\Xone}{X^{(1)}}
\newcommand*{\Xtwo}{X^{(2)}}
\newcommand*{\law}[1]{\mathrm{Law}(#1)}
\newcommand*{\poi}{\mathrm{Pois}}
\newcommand*{\multi}{\mathrm{Multinomial}}
\newcommand*{\diam}{\mathrm{diam}}

\newcommand*{\vstat}{\mathrm{VSTAT}}

\renewcommand{\abstractname}{Abstract.}
\newcommand{\WPP}{\hat W_{p,k}}
\newcommand{\ud}{\mathrm{d}}
\endlocaldefs

\begin{document}

\begin{frontmatter}

	\title{Estimation of Wasserstein distances in the Spiked Transport Model}
	\runtitle{Spiked Transport Model}

	\author{\fnms{Jonathan} \snm{Niles-Weed}\thanksref{t1}\ead[label=jon]{jnw@cims.nyu.edu} 
		~and~
		\fnms{Philippe} \snm{Rigollet}\thanksref{t2}\ead[label=rigollet]{rigollet@math.mit.edu}
	}
	\affiliation{New York University \and Massachusetts Institute of Technology}
    \thankstext{t1}{Part of this research was conducted while at the Institute for Advanced Study. JNW gratefully acknowledges its support.}
	\thankstext{t2}{Supported by NSF awards IIS-BIGDATA-1838071, DMS-1712596 and CCF-TRIPODS-1740751; ONR grant N00014-17- 1-2147.}

	\address{{Jonathan Niles-Weed}\\
		{Courant Institute of Mathematical Sciences} \\
		{New York University}\\
		{251 Mercer Street}\\
		{New York, NY 10012-1185, USA}\\
		\printead{jon}
	}

	\address{{Philippe Rigollet}\\
		{Department of Mathematics} \\
		{Massachusetts Institute of Technology}\\
		{77 Massachusetts Avenue}\\
		{Cambridge, MA 02139-4307, USA}\\
		\printead{rigollet}
	}

	\runauthor{Niles-Weed \& Rigollet}

	\begin{abstract}{}
We propose a new statistical model, the \emph{spiked transport model}, which formalizes the assumption that two probability distributions differ only on a low-dimensional subspace.
We study the minimax rate of estimation for the Wasserstein distance under this model and show that this low-dimensional structure can be exploited to avoid the curse of dimensionality.
As a byproduct of our minimax analysis, we establish a lower bound showing that, in the absence of such structure, the plug-in estimator is nearly rate-optimal for estimating the Wasserstein distance in high dimension.
We also give evidence for a statistical-computational gap and conjecture that any computationally efficient estimator is bound to suffer from the curse of dimensionality.
	\end{abstract}

	\begin{keyword}[class=AMS]
		\kwd[Primary ]{62F99}
		\kwd[; secondary ]{62H99}
	\end{keyword}
	\begin{keyword}[class=KWD]
Wasserstein distance; optimal transport; high-dimensional statistics
	\end{keyword}

\end{frontmatter}

\section{Introduction}
Optimal transport is an increasingly useful toolbox in various data-driven disciplines such as machine learning \citep{AlaGraCut19, PeyCut19,SchHeiBon18,FlaCutCou18,ArjChiBot17,GenPeyCut18,JanCutGra19,MonMulCut16,RolCutPey16,GraJouBer19, StaClaSol17,AlvJaaJeg18,DelGamGor19,CanRos12}, computer graphics \citep{LavClaChi18, SolGoePey15, SolPeyKim16,FeyChaVia17}, statistics~\citep{SegCut15,AhiGouPar18, RigWee18, RigWee19, WeeBer19, ZemPan19, PanZem19, CazSegBig18, RamTriCut17,ClaSol18,TamMun18,KlaTamMun18,BigCazPap17,BarGorLes19,KroSpoSuv19, HutRig19, Le-ParRig19} and the sciences \citep{SchShuTab19, YanDamVen18, delInoLou19, LimLeeCha19}.
 A core primitive of this toolbox is the computation of Wasserstein distances between probability measures, and a natural statistical question is the estimation of Wasserstein distances from data.
 
A key object in this endeavor is the empirical measure $\mu_n$ associated to $\mu$. It is the empirical measure defined by
$$
\mu_n=\frac1n\sum_{i=1}^n\delta_{X_i}, \quad X_i \sim \mu \quad \text{i.i.d}\,.
$$
Owing to their flexibility, Wasserstein distances are notoriously hard to estimate in high dimension since in such cases, the empirical distribution is a poor proxy for the underlying distribution of interest.
Indeed, for $d$-dimensional distributions, Wasserstein distances between empirical measures generally converge at the slow rate $n^{-1/d}$~\citep{BoiLe-14,FouGui15,WeeBac18,DerSchSch13,Dud69} and thus suffer from the curse of dimensionality.
For example, the following behavior is typical.
\begin{proposition}\label{prop:empirical}
Let $\mu$ be a probability measure on $[-1, 1]^d$. If $\mu_n$ is an empirical measure comprising $n$ i.i.d.\ samples from $\mu$, then for any $p \in [1, \infty)$,
\begin{equation*}
\E W_{p}\left(\mu_n, \mu\right) \leq \rated{n} \defeq c_{p} \sqrt d \left\{
\begin{array}{ll}
n^{-1/2p} & \text{if $d < 2p$} \\
n^{-1/2p} (\log n)^{1/p} & \text{if $d = 2p$} \\
n^{-1/d} & \text{if $d > 2p$}
\end{array}
\right.
\end{equation*}
\end{proposition}
In many settings, this bound is known to be tight up to logarithmic factors. In fact, the rate in \cref{prop:empirical} has been shown to be essentially minimax optimal for the problem of estimating $\mu$ in Wasserstein distance~\citep{SinPoc18}, using any estimator, not necessarily the empirical distribution~$\mu_n$. 

Since $W_p$ satisfies the triangle inequality, \cref{prop:empirical} readily yields that $W_p(\muone, \mutwo)$ for two probability measures $\muone, \mutwo$ on $[-1, 1]^d$ can be estimated at the rate $n^{-1/d}$ by the plug-in estimator $W_p(\muone_n, \mutwo_n)$ when $d > 2p$, and a lower bound of the same order can be shown for the plug-in estimator when $\muone$ and $\mutwo$ are, for example, the uniform measure on $[-1, 1]^d$. However, while the above results give a strong indication that the Wasserstein distance $W_p(\muone, \mutwo)$ itself is also hard to estimate in high dimension, they do not preclude the existence of estimators that are better than $W_p(\muone_n, \mutwo_n)$.
Indeed, until recently, the best known lower bound for the problem of estimating the distance itself was of order $n^{-3/2d}$~\citep{Do-NguNgu11}.
A concurrent and independent result~\citep{Lia19} closes this gap for $p = 1$ and indicates that estimating the $W_1$ itself is essentially as hard as estimating the measure itself. Our results show that, in fact, estimating the distance $W_p(\muone, \mutwo)$ is essentially as hard as estimating a measure $\mu$ in $W_p$-distance, for any $p\ge 1$. As a result, any estimator of the distance itself must suffer from the curse of dimensionality.

One goal of~\emph{statistical optimal transport} is to develop new models and methods that overcome this curse of dimensionality by leveraging plausible structure in the problem. Early contributions in this direction include assuming smoothness~\citep{HutRig19,WeeBer19,Lia19} or sparsity~\citep{ForHutNit19}. In this work, we propose a new model, called the \emph{spiked transport model}, to formalize the assumption that two distributions differ only on a low dimensional subspace of $\R^d$. Such an assumption forms the basis of several popular alternatives to Wasserstein distances such as the Sliced Wasserstein distance~\citep{RabPeyDel11} or random one-dimensional projections~\citep{PitKokDah07}. More recently, several numerical methods that exploit a form of low-dimensional structure were proposed together with illustration of their good numerical performance~\citep{KolNadSim19,DesHuSun19,PatCut19}.

To exploit the low-dimensional structure of the spiked transport model, we consider a standard method in statistics often called ``projection pursuit"~\citep{Kru69,Kru72,FriTuk74}. This general method aims to alleviate the challenges of high-dimensional data analysis by considering low-dimensional projections of the dataset that reveal ``interesting'' features of the data. We show that a suitable instantiation of this method to the present problem, which we call \emph{Wasserstein Projection Pursuit} (WPP), leads to near optimal rates of estimation of the Wasserstein distance in the spiked transport model and permits to alleviate the curse of dimensionality from which the plug-in estimator suffers.

While our results establish a clear statistical picture, it is unclear how to implement WPP efficiently. An efficient relaxation of this estimator was recently proposed by~\citet{PatCut19}, and a natural question is to analyze its performance in the spiked transport model. Instead of pursuing this direction we bring strong evidence that, in fact, \emph{no} computationally efficient estimator is likely to be able to take advantage of the low-dimensional structure inherent to the spiked transport model. Our computational lower bounds come from the well-established statistical query framework~\citep{Kea98}. In particular, they indicate a fundamental tradeoff between statistical and computational efficiency~\citep{BerRig13,BreBreHul18,BanPerWei18,MaWu15,CaiLiaRak17}: computationally efficient methods to estimate Wasserstein distances, are bound to suffer the curse of dimensionality.

The rest of this paper is organized as follows.
We introduce our model and main results in \cref{sec:model,sec:main}.
In \cref{sec:low-dim}, we define a low-dimensional version of the Wasserstein distance and establish its connection to the spiked transport model.
\Cref{sec:concentration} proves the equivalence between transportation inequalities and subgaussian concentration properties of the Wasserstein distance.
In \cref{sec:upper} we propose and analyze an estimator for the Wasserstein distance under the spiked transport model.
We establish a minimax lower bound in \cref{sec:lower}.
Finally, we prove a statistical query bound on the performance of efficient estimators in \cref{sec:comp}.
Supplementary proofs and lemmas appear in the appendices.

\medskip

\noindent{\sc Notation.}

We denote by $\|\cdot\|$ the Euclidean norm on $\RR^d$.
The symbols~$\op{\cdot}$ and $\frob{\cdot}$ denote the operator norm and Frobenius norm, respectively.
If $X$ is a random variable on $\RR$, we let $\|X\|_p \defeq (\E |X|^p)^{1/p}$.
Throughout, we use $c$ and $C$ to denote positive constants whose value may change from line to line, and we use subscripts to indicate when these constants depend on other parameters.
We write $a \lesssim b$ if $a \leq C b$ holds for a universal positive constant $C$.

\section{Model and methods}\label{sec:model}
In this section, we describe the spiked transport model and Wasserstein projection pursuit.

\subsection{Wasserstein distances}
Given two probability measures $\mu$ and $\nu$ on~$\R^d$, let $\Gamma_{\mu,\nu}$ denote the set of \emph{couplings} between $\mu$ and $\nu$ so that $\gamma \in \Gamma_{\mu,\nu}$ iff $\gamma(U\times \R^d)=\mu(U)$ and $\gamma(\R^d\times V) = \nu(V)$.

For any $p \ge 1$, the \emph{$p$-Wasserstein distance} $W_p$ between   $\mu$ and $\nu$   is defined as
\begin{equation}\label{eq:wp}
W_p(\mu,\nu) :=\inf_{\gamma\in\Gamma_{\mu,\nu}} \left(\int_{\R^d\times \R^d} \|x-y\|^p\, \ud\gamma(x,y)\right)^{1/p}\,.
\end{equation}

The definition of Wasserstein distances may be extended to measures defined on general metric spaces but such extensions are beyond the scope of this paper. We refer the reader to~\citet{Vil09} for a comprehensive treatment.

\subsection{Spiked transport model}

We introduce a new model that induces a low-dimensional structure on the optimal transport between two measures $\muone$ and $\mutwo$ over $\R^d$. To that end, fix a subspace $\cU \subseteq \RR^d$ of dimension $k \ll d$ and let $X^{(1)}, X^{(2)} \in \cU$ be two random variables with arbitrary distributions. Next, let $Z$ be a third random variable, independent of $(\Xone,\Xtwo)$ and such that $Z$ is supported on the orthogonal complement $\cU^\perp$ of $\cU$. Finally, let
\begin{equation}\label{eq:spiked_transport}
\begin{aligned}
\muone & \defeq \law{\Xone + Z} \\
\mutwo & \defeq \law{\Xtwo + Z}\,.
\end{aligned}
\end{equation}
Though $\muone$ and $\mutwo$ are high-dimensional distributions, they differ only on the low-dimensional subspace $\cU$. 
Borrowing terminology from principal component analysis~\citep{Joh01}, we say that the pair $(\muone,\mutwo)$ satisfies the  \emph{spiked transport model} and we call $\cU$ the \emph{spike}.

We pose the following question: given $n$ independent observations from both $\muone$ and $\mutwo$, is it possible to estimate the Wasserstein distance between them at a rate faster than $n^{-1/d}$?

\subsection{Concentration assumptions}
In order to establish sharp statistical results for estimation of the Wasserstein distance, it is necessary to adopt smoothness and decay assumptions on the measures in question~\citep[see, e.g.,][]{BobLed14,FouGui15}.
We focus on a family of such conditions known as \emph{transport inequalities}, the study of which is a central object in the theory of concentration of measure~\citep{Led01}. 

A probability measure $\mu$ on $\RR^d$ is said to satisfy the $\tp{\sigma^2}$ transport inequality  if
\begin{equation}
\label{eq:tpineq}
W_p(\nu, \mu) \leq \sqrt{2 \sigma^2 D(\nu \| \mu)}
\end{equation}
for all probability measures $\nu$ on $\RR^d$.

These inequalities interpolate between several well known assumptions in high-dimensional probability.
For example inequality $\tone{\sigma^2}$ is essentially equivalent to the assertion that $\mu$ is subgaussian, and $T_2(\sigma^2)$ is implied by (and often equivalent to) a stronger log-Sobolev inequality~\citep{GozLeo10}.

Our main results on the estimation of $W_p(\muone, \mutwo)$ are established under the assumption that both  $\mu^{(1)}$ and $\mu^{(2)}$ satisfy a $T_p(\sigma^2)$ transport inequality.
We show in Section~\ref{sec:concentration} that~\eqref{eq:tpineq} is precisely equivalent to requiring that the random variable $W_p(\mu_n, \mu)$  is subgaussian.

\subsection{Wasserstein projection pursuit}
To take advantage of the spiked transport model, we employ a natural estimation method that we call Wasserstein Projection Pursuit (WPP). 

Let $\mu$ and $\nu$ be two probability distributions on $\R^d$. Given a $k \times d$ matrix $U$ with orthonormal rows, let  $\mu_U$ (resp. $\nu_U)$ denote the distribution of $U Y$ where $Y \sim \mu$ (resp. $Y\sim \nu$). We define
\begin{equation*}
\tilde W_{p,k}(\mu,\nu)  \defeq \max_{U \in \stie} W_p(\mu_U, \nu_U)\,,
\end{equation*}
where the maximization is taken over the \emph{Stiefel manifold} $\stie$ of $k \times d$ matrices with orthonormal rows.

Given empirical measures $\muone_n$ and $\mutwo_n$ associated to $\muone$ and $\mutwo$ that satisfy the spiked transport model, we propose the following WPP estimator of $W_p(\muone, \mutwo)$:
$$
\WPP=\tilde W_{p,k}(\muone_n,\mutwo_n)\,,
$$
In the next section, we show that this estimator is near-minimax-optimal.
\section{Main results}\label{sec:main}
As a theoretical justification for Wasserstein projection pursuit, we prove that our procedure successfully avoids the curse of dimensionality under the spiked transport model. Our results primarily focus on the estimation of the the Wasserstein distance itself but we also obtain as a byproduct of Wasserstein projection pursuit an estimator for the spike~$\cU$ using standard perturbation results.

\subsection{Estimation of the Wasserstein distance}
The following theorem shows that Wasserstein projection pursuit takes advantage of the low-dimensional structure of the spiked transport model when estimating the Wasserstein distance.
\begin{theorem}\label{thm:intro-upper}
Let $(\mu^{(1)},\mu^{(2)})$ satisfy the spiked transport model~\eqref{eq:spiked_transport}. For any $p \in [1,2]$, if $\mu^{(1)}$ and $\mu^{(2)}$ satisfy the $\tp{\sigma^2}$ transport inequality, then the WPP estimator $\WPP$
satisfies
\begin{equation*}
\E \big|\WPP - W_p(\mu^{(1)}, \mu^{(2)})\big| \leq c_k \cdot \sigma \Big(\rate{n} + \sqrt{\frac{d \log n}{n}}\Big)\,.
\end{equation*}
\end{theorem}
Strikingly, the rate $\rated{n}$ achieved by the na\"ive plug-in estimator (see \cref{prop:empirical}) has been replaced by $\rate{n}$---in other words, this estimator enjoys the rate typical for $k$-dimensional rather than $d$-dimensional measures.
The only dependence on the ambient dimension is in the second term, which is of lower order than the first whenever $p > 1$ or $k > 2$.
A more general version of this theorem appears in \cref{sec:upper}.

\subsection{Estimation of the spike}
We  show that if the distance between $\muone$ and $\mutwo$ is large enough, Wasserstein projection pursuit recovers the subspace $\cU$.
For simplicity, we state here the result when $k = 1$ and defer the full version to \cref{sec:upper}.
\begin{theorem}
Let $(\mu^{(1)},\mu^{(2)})$ satisfy the spiked transport model with $k=1$ and let $\cU$ be spanned by the unit vector $u \in \R^d$. Fix $p \in [1,2]$ and assume that $\mu^{(1)}$ and $\mu^{(2)}$ satisfy the $\tp{\sigma^2}$ transport inequality and that $W_p(\muone, \mutwo) \gtrsim 1$. 
Then the estimator
\begin{equation*}
\hat u \defeq \argmax_{v \in \RR^d, \|v\| = 1}  W_p(\muone_v, \mutwo_v)
\end{equation*}
satisfies
\begin{equation*}
\E \sin^2\big(\measuredangle(\hat u, u)\big) \lesssim \sigma\cdot \Big(n^{-1/2p} + \sqrt{\frac{d \log n}{n}}\Big)\,.
\end{equation*}
\end{theorem}

\subsection{Lower bounds}
To show that \cref{thm:intro-upper} has the right dependence on $n$ and $d$, we exhibit two lower bounds, which imply that neither term in \cref{thm:intro-upper} can be avoided.

To show the optimality of the first term, we define
\begin{equation*}
\ratedlb{n} \defeq c_{p,d}  \left\{
\begin{array}{ll}
n^{-1/2p} & \text{if $d < 2p$} \\
n^{-1/2p}  & \text{if $d = 2p$} \\
(n \log n)^{-1/d} & \text{if $d > 2p$}
\end{array}
\right.
\end{equation*}

\begin{theorem}\label{intro-lower}
Fix $p \geq 1$. For any estimator $\hat W$, there exists a pair of measures $\mu$ and $\nu$ supported on $[0, 1]^d$ such
\begin{equation*}
\E |\hat W - W_p(\muone, \mutwo)| \geq \ratedlb{n}\,.
\end{equation*}
\end{theorem}
This lower bound readily implies that the plug-in estimator for the Wasserstein distance is optimal up to logarithmic factors.
By embedding $[0, 1]^k$ into $[0, 1]^d$, this result likewise implies that the term $\rate{n}$ in \cref{thm:intro-upper} is essentially optimal.
A proof appears in \cref{sec:lower}

Independently, \citet{Lia19} recently obtained a similar result in the case $p=1$. More specifically, he proved that when $d \ge 2$, for any estimator $\hat W$, there exist probability measures $\muone$ and $\mutwo$ such that the following lower bound holds:
$$
\E |\hat W - W_1(\muone, \mutwo)| \gtrsim \frac{\log \log n}{\log n} n^{-1/d}
$$
In particular, while our lower bound is slightly stronger and holds for all $p \geq 1$, both our result and that of \citet{Lia19} fail to match the na\"ive upper bound of order $n^{1/d}$ by logarithmic factors when $d$ is large. The presence of a logarithmic factor in our lower bound comes from a reduction to estimating the total variation distance. In that case, as in several other instances of functional estimation problems, the presence of this factor is, in fact, optimal, and has been dubbed \emph{sample size enlargement}~\citep{JiaHanWei18}. Closing this gap in the context of estimation of the Wasserstein distance is an interesting and fundamental question.

The only appearance of the ambient dimension $d$ is in the second term of \cref{thm:intro-upper}.
The following theorem shows that this dependence cannot be eliminated, even when $k = 1$.
\begin{theorem}\label{deviation_lb}
Let $p \in [1, 2]$ and $\sigma > 0$, and assume $k=1$. For all estimators $\hat W$, there exists a pair of measures $\muone$ and $\mutwo$ satisfying the spiked transport model and $T_p(\sigma^2)$ such that
\begin{equation*}
\E |\hat W - W_p(\muone, \mutwo)| \gtrsim \sigma \sqrt{\frac{d} n}\,.
\end{equation*}
\end{theorem}
The proof of \cref{deviation_lb} is deferred to the appendix. For problems where $d$ is large, dependence on $d$ may be a crippling limitation. In that case, we conjecture that assuming a sparse spike, in the same spirit as sparse PCA, can mitigate this effect and bring interpretability to the estimated spike.

\subsection{A computational-statistical gap}
The WPP estimator achieving the rate in \cref{thm:intro-upper} is computationally expensive to implement, which raises the question of whether an efficient estimator exists achieving the same rate.
We give evidence in the form of a statistical query lower bound that no such estimator exists.
The statistical query model considers algorithms with access to an oracle $\vstat(t)$, where $t > 0$ is a parameter which plays the role of sample size.
We show that any such algorithm for estimating the Wasserstein distance needs an \emph{exponential} number of queries to an oracle with \emph{exponential} sample size parameter, even under the spiked transport model.
By contrast, \cref{thm:intro-upper} implies that a non-efficient estimator needs a number of samples only polynomial in the dimension.
\begin{theorem}
Let $p \in [1, 2]$, and consider probability measures $\muone$ and $\mutwo$ satisfying the spiked transport model.
There exists a positive constant~$c$ such that any statistical query algorithm which estimates $W_p(\muone, \mutwo)$ to accuracy $1/\mathrm{poly}(d)$ with probability at least $2/3$ requires at least $2^{cd}$ queries to $\vstat(2^{cd})$.
\end{theorem}

\section{Low-dimensional Wasserstein distances}\label{sec:low-dim}

Motivated by projection pursuit, we define the following version of the Wasserstein distance which measures the discrepancy between low-dimensional projections of the measures.
\begin{defin}
For $k \in [d]$, the \emph{$k$-dimensional Wasserstein distance} between $\muone$ and $\mutwo$ is
\begin{equation*}
\wpk(\muone, \mutwo) \defeq \sup_{U \in \stie} W_p(\muone_U, \mutwo_U)\,.
\end{equation*}
\end{defin}
This definition has been proposed independently and concurrently by a number of other recent works~\citep{KolNadSim19,DesHuSun19,PatCut19}.

We will use throughout the following basic fact about the $k$-dimensional Wasserstein distance~\citep[Proposition 1]{PatCut19}.

\begin{proposition}\label{prop:triangle}
$\wpk$ is a metric on the set of probability measures over $\RR^d$ with finite $p$th moment.
\end{proposition}

The definition of $\wpk$ is chosen so that, under the spiked transport model~\eqref{eq:spiked_transport}, the $k$-dimensional Wasserstein distance agrees with the normal Wasserstein distance.

\begin{proposition}\label{prop:wpk-is-wp}
Under the spiked transport model~\eqref{eq:spiked_transport},
\begin{equation*}
\wpk(\muone, \mutwo) = W_p(\muone, \mutwo)\,.
\end{equation*}
\end{proposition}

\Cref{prop:wpk-is-wp} follows from the following statement, which pertains to distributions that are allowed to have a different component on the space orthogonal to the spike $\cU$. 
Suppose $\nuone$ and $\nutwo$ satisfy
\begin{equation}\label{eq:almost-spike}
\begin{aligned}
\nuone & = \law{X^{(1)} + Z^{(1)}} \\
\nutwo & = \law{X^{(2)} + Z^{(2)}}\,,
\end{aligned}
\end{equation}
where as before $X^{(1)}$ and $X^{(2)}$ are supported on a subspace $\cU$ and $Z^{(1)}$ and $Z^{(2)}$ are supported on its orthogonal complement $\cU^\perp$, and where we assume that $X^{(i)}$ and $Z^{(i)}$ are independent for $i \in \{1, 2\}$. Note that unlike in the spiked transport model, the components $Z^{(1)}$ and $Z^{(2)}$ on the orthogonal complement of $\cU$ need not be identical. The following result shows that under this relaxed model, the $k$-dimensional Wasserstein distance between $\nuone$ and $\nutwo$ still captures the true Wasserstein distance between the distributions as long as the distributions of $Z^{(1)}$ and $Z^{(2)}$ are sufficiently close.

\begin{proposition}\label{prop:almost-spike}
Under the relaxed spiked transport model~\eqref{eq:almost-spike},
\begin{equation*}
|\wpk(\nuone, \nutwo) - W_p(\nuone, \nutwo)| \leq W_p(\law{Z^{(1)}}, \law{Z^{(2)}})\,.
\end{equation*}
\end{proposition}

\section{Concentration}\label{sec:concentration}
A key step to establish the upper bound of \cref{sec:upper} consists in establishing good concentration properties for the Wasserstein distance between a measure and its empirical counterpart.
The main assumption we adopt is that the measures in question satisfy a \emph{transport inequality}.
Since the pioneering work of \cite{Mar96, Mar96a} and \cite{Tal96}, transport inequalities have played a central role in the analysis of the concentration properties of high-dimensional measures.

We require two definitions.
\begin{defin}
Given a Polish space~ $\cX$ equipped with a metric~$\rho$, denote by~$\cP(\cX)$ 
the space of all Borel probability measures $\cX$.
Let $\cP_p(\cX) \defeq \{\mu \in \cP(\cX) : \int \rho(x, \cdot)^p \dd \mu(x) < \infty\}$.

A measure $\mu \in \cP_p(\cX)$ satisfies the \emph{$\tp{\sigma^2}$ inequality} for some $\sigma > 0$ if
\begin{equation*}
W_p(\nu, \mu) \leq \sqrt{2 \sigma^2 D(\nu \| \mu)} \quad \quad \forall \nu \in \cP_p(\cX)\,,
\end{equation*}
where $W_p$ is the Wasserstein-$p$ distance on $(\cX,\rho)$ and $D$ is the Kullback-Leibler divergence.
\end{defin}
\begin{defin}\label{def:subg}
A random variable $X$ on $\RR$ is $\sigma^2$-subgaussian if
\begin{equation*}
\E e^{\lambda(X - \E X)} \leq e^{\lambda^2 \sigma^2/2} \qquad \forall \lambda \in \RR\,.
\end{equation*}
\end{defin}

In this section, we present a surprisingly simple equivalence between transport inequalities and subgaussian concentration for the Wasserstein distance.
The essence of this result is present in the works of \citeauthor{GozLeo10}~\citep[see][]{GozLeo10,Goz09,GozLeo07}, and similar bounds have been obtained by~\citet{BolGuiVil07}. Nevertheless, we could not find this simple fact stated in a form suitable for our purposes in the literature. For any measure $\mu$, recall that the random measure $\mu_n \defeq \frac 1n \sum_{i=1}^n \delta_{X_i}$, where $X_i \sim \mu$ i.i.d., denotes its associated empirical measure.
\begin{theorem}\label{thm:concentration_is_tp}
Let $p \in [1,2]$. A measure~$\mu \in \cP_p(\cX)$ satisfies $\tp{\sigma^2}$ if and only if the random variable $W_p(\mu_n, \mu)$ is $\sigma^2/n$-subgaussian for all $n$.
\end{theorem}

Because $\tpp{\sigma^2}$ implies $\tp{\sigma^2}$ when $p' \geq p$, \cref{thm:concentration_is_tp} also implies that a measure satisfying $\tpp{\sigma^2}$ has good concentration for $W_p$ if $p \leq p'$.
In the opposite direction, if $p > p'$, then satisfying $\tpp{\sigma^2}$ still yields a weaker concentration bound.
A modification of the proof of \cref{thm:concentration_is_tp} yields the following result.

\begin{theorem}\label{thm:p-pprime-concentration}
Let $p' \in [1, 2]$ and $p \geq 1$.
If $\mu \in \cP_{p'}(\cX)$ satisfies $\tpp{\sigma^2}$, then $W_p(\mu_n, \mu)$ is $\sigma^2/n^{1 - \left(\frac{2}{p'} - \frac 2 p\right)_+}$ subgaussian.
\end{theorem}

The conclusion of \cref{thm:p-pprime-concentration} is interesting whenever $ \left(\frac{2}{p'} - \frac 2 p\right)_+ < 1$. For example, if we assume merely that $\mu$ satisfies $\tone{\sigma^2}$, \cref{thm:p-pprime-concentration} only yields a nontrivial concentration result for $W_p(\mu_n, \mu)$ when $p < 2$; by contrast, if $\mu$ satisfies $T_2(\sigma^2)$, then \cref{thm:p-pprime-concentration} implies a concentration result for $W_p(\mu_n, \mu)$ for all $p < \infty$.

In \cref{sec:upper}, we require concentration properties not of the Wasserstein distance itself but of the $k$-dimensional Wasserstein distance.
The following result shows that low-dimensional projections inherit the concentration properties of the $d$-dimensional measure.
\begin{proposition}
Let $U \in \stie$.
For any $p \in [1, 2]$ and $\sigma > 0$, if $\mu$ satisfies $\tp{\sigma^2}$, then so does $\mu_U$.
\end{proposition}
\begin{proof}
The projection $x \mapsto Ux$ is a contraction.
The result then follows from \citet[Corollary 20]{Goz07} \citep[see also][]{Mau91}.
\end{proof}

We conclude this section by giving some simple conditions under which the $\tone{\sigma^2}$ inequality is satisfied.
The following characterization is well known.
Denote by $\lip(\cX)$ the space of all functions $f: \cX \to \RR$ satisfying $|f(x) - f(y)| \leq d(x, y)$ for all $x, y \in \RR$.
\begin{proposition}[{\citealp[Theorem 1.3]{BobGot99}}]\label{prop:bobgot}
A measure $\mu \in \cP_1(\cX)$ satisfies $\tone{\sigma^2}$ if and only if $f(X)$ is $\sigma^2$-subgaussian for all $f \in \lip(\cX)$.
\end{proposition}

It is common to extend \cref{def:subg} to random vectors as follows.
\begin{defin}
A random variable $X$ on $\RR^d$ is \emph{$\sigma^2$-subgaussian} if $u^\top X$ is $\sigma^2$-subgaussian for all $u \in \RR^d$ satisfying $\|u\| = 1$.
\end{defin}

Subgaussian random vectors yield a large collection of random variables satisfying a $T_1$ inequality.
\begin{lemma}\label{lem:t1_is_subg}
If $\mu$ on $\RR^k$ satisfies $\tp{\sigma^2}$ for any $p \geq 1$, then $X \sim \mu$ is $\sigma^2$-subgaussian.
Conversely, if $X \sim \mu$ on $\RR^k$ is $\sigma^2$-subgaussian, then $\mu$ satisfies~$\tone{C k \sigma^2}$ for a universal constant $C > 0$.
\end{lemma}
If the entries of $X$ are independent, then the result holds with $C = 1$ by a result of \citet{Mar96}.
The presence of the factor $k$ in the converse statement is unavoidable; unlike $T_2$ inequalities, $T_1$ inequalities do not exhibit dimension-free concentration~\citep{GozLeo10}.

\section{Upper bounds}\label{sec:upper}
In this section, we establish that under the spiked transport model, Wasserstein projection pursuit produces a significantly more accurate estimate of the Wasserstein distance than the plug-in estimator.

Let $\muone$ and $\mutwo$ be two measures generated according to the spiked transport model~\eqref{eq:spiked_transport}.
For $i \in \{1, 2\}$, we let $\mui_n \defeq \frac 1n \sum_{j=1}^n \delta_{X^{(i)}_j}$, where $X^{(i)}_j \sim \mui$ are i.i.d.
We define
\begin{equation*}
\hat W_{p,k} \defeq \wpk(\muone_n, \mutwo_n)\,.
\end{equation*}

Our main upper bound shows that $\hat W_{p,k}$ converges to the true Wasserstein distance $W_p(\mu, \nu)$ at a rate much faster than $n^{-1/d}$.

\begin{theorem}\label{thm:main}
Let $p' \in [1, 2]$ and $p \geq 1$.
Under the spiked transport model, if $\muone$ and $\mutwo$ satisfy $\tpp{\sigma^2}$, then
\begin{equation*}
\E |\hat W_{p,k} - W_p(\muone, \mutwo)| \lesssim \sigma \left(\rate{n} + c_p \cdot n^{(\frac{1}{p'} - \frac{1}{p})_+} \sqrt\frac{dk \log n}{n}\right)\,.
\end{equation*}
\end{theorem}

\Cref{thm:main} can also be extended to the misspecified model proposed in~\eqref{eq:almost-spike}

\begin{theorem}\label{thm:almost}
Let $p' \in [1, 2]$ and $p \geq 1$.
Under the relaxed spiked transport model \eqref{eq:almost-spike}, if $\nuone$ and $\nutwo$ satisfy $\tpp{\sigma^2}$, then
\begin{equation*}
\E |\hat W_{p,k} - W_p(\nuone, \nutwo)| \lesssim \sigma \left(\rate{n} + c_p \cdot n^{(\frac{1}{p'} - \frac{1}{p})_+} \sqrt\frac{dk \log n}{n}\right) + \ep\,,
\end{equation*}
where $\ep = W_p(\law{Z^{(1)}}, \law{Z^{(2)}})$.
\end{theorem}
\Cref{thm:almost}, which follows almost immediately from the proof of \cref{thm:main}, establishes that Wasserstein projection pursuit brings statistical benefits even in the situation where the spiked transport model holds only approximately.

\Cref{thm:main} follows from the following two propositions.
We first show that the quality of the proposed estimator $\hat W_{p,k}$ can be bounded by the sum of two terms depending only on $\muone_n$ and $\mutwo_n$ individually.
\begin{proposition}\label{prop:split_obj}
\begin{equation*}
\E| \hat W_p - W_p(\muone, \mutwo)| \leq \E \wpk(\muone,  \muone_n) + \E \wpk(\mutwo, \mutwo_n)
\end{equation*}
\end{proposition}
\begin{proof}
Since $\hat W_p = \wpk(\muone_n, \mutwo_n)$ and $W_p(\muone, \mutwo) = \wpk(\muone, \mutwo)$, the claim is immediate from \cref{prop:triangle}.
\end{proof}

The following proposition allows us to bound both terms of \cref{prop:split_obj} by the desired quantity.
\begin{proposition}\label{prop:supremum}
Let $p' \in [1, 2]$ and $p \geq 1$. If $\mu$ satisfies $\tpp{\sigma^2}$, then
\begin{equation*}
\E \wpk(\mu, \mu_n) \lesssim \sigma \left(r_{p,k}(n) + c
_p \cdot n^{(\frac{1}{p'} - \frac{1}{p})_+}\sqrt{\frac{dk \log n}{n}}\right)\,.
\end{equation*}
\end{proposition}
\begin{proof}
Wasserstein distances are invariant under translating both measures by the same vector. Therefore, we can assume without loss of generality that $\mu$ has mean $0$.
Likewise, by homogeneity, we assume $\sigma = 1$.

Let $Z_U \defeq W_p(\mu_U, (\mu_n)_U)$.
We first show that the process $Z_U$ is Lipschitz.
\begin{lemma}
There exists a random variable $L$ such that for all $U, V \in \cV_k(\RR^d)$,
\begin{equation*}
|Z_U - Z_V| \leq L \op{U - V}
\end{equation*}
and $\E L \lesssim \sqrt{dp}$.
\end{lemma}
\begin{proof}
Let $X \sim \mu$.
Then
\begin{align*}
|Z_U - Z_V| & \leq W_{p}(\mu_U, \mu_V) + W_{p}((\mu_n)_U, (\mu_n)_V) \\
& \leq (\E \|(U-V) X\|^p)^{1/p} + \Big(\frac 1 n\sum_{i=1}^n \|(U-V) X_i\|^p\Big)^{1/p} \\
& \leq \op{U-V} \Big((\E \|X\|^p)^{1/p} + \Big(\frac 1 n\sum_{i=1}^n \|X_i\|^p\Big)^{1/p}\Big)\,.
\end{align*}

We obtain that
\begin{equation*}
|Z_U - Z_V| \leq L \op{U-V}
\end{equation*}
where $L = (\E \|X\|^p)^{1/p} + \left(\frac 1 n\sum_{i=1}^n \|X_i\|^p\right)^{1/p}$. 
By Jensen's inequality, we have $\E L \leq 2 (\E \|X\|^p)^{1/p}$.
Together with \cref{prop:subg_vec_norm}, it yields the claim.
\end{proof}

By \cref{thm:p-pprime-concentration}, for all $U \in \cV_k(\RR^d)$, the random variable $Z_U$ is $n^{-1 + \left(\frac{2}{p'} - \frac 2 p\right)_+}$ subgaussian.
Therefore, by a standard $\ep$-net argument, if we denote by $\cN(\cV_k, \ep, \op{\cdot})$ the covering number of $\cV_k$ with respect to the operator norm, we obtain
\begin{equation*}
\E \sup_{U \in \stie} (Z_U - \E Z_U) \lesssim \inf_{\ep > 0} \left\{\ep \E L+  n^{(\frac 1p - \frac{1}{p'})_+}\sqrt{\frac{\log \cN(\cV_k, \ep, \op{\cdot})}{n}}\right\}\,.
\end{equation*}
\Cref{lem:vk_op_covering} shows that there exists a universal constant $c$ such that $\log \cN(\cV_k, \ep, \op{\cdot}) \leq dk \log \frac {c\sqrt k} \ep$ for $\ep \in (0, 1]$.
Choosing $\ep = \sqrt{k/n}$ yields 
\begin{align*}
\E \sup_{U \in \stie} (Z_U - \E Z_U) & \lesssim \sqrt{\frac{dkp}{n}} + n^{(\frac{1}{p'} - \frac{1}{p})_+}\sqrt{\frac{dk \log n}{n}} \\
& \leq c_{p} \cdot  n^{(\frac{1}{p'} - \frac{1}{p})_+}\sqrt{\frac{dk \log n}{n}}\,.
\end{align*}

Applying \cref{subgaussian_wp_rate} yields
\begin{align*}
\E \sup_{U \in \stie} W_p(\mu_U, (\mu_n)_U) & \leq  \sup_{U \in \stie} \E W_p(\mu_U, (\mu_n)_U)+\E \sup_{U \in \stie} (Z_U - \E Z_U) \\
& \lesssim r_{p,k}(n) + c_p\cdot n^{(\frac{1}{p'} - \frac{1}{p})_+} \sqrt{\frac{dk \log n}{n}}\,,
\end{align*}
as claimed.
\end{proof}

We also obtain a Davis-Kahan-type theorem on subspace recovery.
Given two subspaces $\cU_1$ and $\cU_2$, the \emph{minimal angle}~\citep{Deu95,Dix49} between them is defined to be
\begin{equation*}
\measuredangle(\cU_1, \cU_2) \defeq \arccos \left(\sup_{u_1 \in \cU_1, u_2 \in \cU_2} \frac{u_1^\top u_2}{\|u_1\| \|u_2\|}\right)\,.
\end{equation*}
If $\measuredangle(\cU_1, \cU_2) = 0$, then $\cU_1 \cap \cU_2 \neq \{0\}$, so that $\cU_1$ and $\cU_2$ are at least partially aligned.
In the important special case that $\cU_1$ and $\cU_2$ are each one dimensional, this definition reduces to the angle between the subspaces.

The following result indicates that as long as $\muone$ and $\mutwo$ are well separated, Wasserstein projection pursuit also yields a subspace with at least partial alignment to $\cU$.

\begin{theorem}\label{thm:sin}
Let $p' \in [1, 2]$ and $p \geq 1$.
Assume that $\muone$ and $\mutwo$ satisfy the spiked transport model and $\tpp{\sigma^2}$
Let $\hat \cU \defeq \mathrm{span}(\hat U)$, where 
\begin{equation*}
\hat U \defeq \argmax_{U \in \stie} W_p\big(({\mu}_n^{(1)})_U, ({\mu}_n^{(2)})_U\big)\,.
\end{equation*}
Then
\begin{equation*}
\E \sin^2\big(\measuredangle(\hat \cU, \cU)\big) \lesssim \frac{\sigma \left(\rate{n} + c_p \cdot n^{(\frac{1}{p'} - \frac{1}{p})_+} \sqrt\frac{dk \log n}{n}\right)}{W_p(\muone, \mutwo)}\,.
\end{equation*}
\end{theorem}
A proof of \cref{thm:sin} appears in the appendix. Note that the a bound on the minimal angle is a rather weak guarantee. Indeed, $\measuredangle(\hat \cU, \cV)\to 0$ implies that the subspaces $\hat \cU$ and $\cU$ share at least a common line asymptotically but not more. When $k=1$, this ensures recovery of the subspace, but this no longer holds true for higher dimensional spikes. In retrospect such a guarantee is all that we can hope for under the mere assumption that $W_p(\muone, \mutwo)>0$. Indeed, it may be the case that these distributions differ only on a one dimensional space. Stronger guarantees may be achieved by assuming that $W_p(\muone_{V}, \mutwo_{V})>0$ for a large family of $V$, but we leave them for future research.

\section{A lower bound on estimating the Wasserstein distance}\label{sec:lower}
In this section, we prove that the rate $\rated{n}$ is optimal for estimating the Wasserstein distance, up to logarithmic factors.
The core idea of our lower bound is to relate estimating the Wasserstein distance to the problem of estimating total variation distance, sharp rates for which are known~\citep{ValVal11,JiaHanWei18}.
To obtain sufficient control over the Wasserstein distance as a function of total variation, we prove a refined bound incorporating both total variation and the $\chi^2$ divergence (\cref{random_injection}).
We then show a modified lower bound (\cref{tv_lb}) for a testing problem involving the total variation distance over the class of distributions on $[m]$ close to the uniform measure in $\chi^2$ divergence.

In the interest of generality, we formulate our results for any compact metric space $\cX$ whose covering numbers satisfy
\begin{equation}\label{covering-assumption}
c \ep^{-d} \leq \cN(\cX, \ep) \leq C \ep^{-d}
\end{equation}
for all $\ep \leq \diam(\cX)$.
This condition clearly holds for compact subsets of $\RR^d$ and more generally for metric spaces with Minkowski dimension $d$.
We adopt the assumption $\diam(\cX) = 1$ without loss of generality.

Let $\cP$ be the set of distributions supported on $\cX$ and let $R(n, \cP)$ denote the minimax risk over $\cP$,
\begin{equation*}
R(n, \cP) \defeq \inf_{\hat W} \sup_{\mu, \nu \in \cP} \E_{\mu, \nu} |\hat W - W_p(\mu, \nu)|\,.
\end{equation*}

The bound $R(n, \cP) \gtrsim n^{-1/2p}$ is an almost trivial consequence of the fact that the distribution $\frac 12 \delta_{-1} + \frac 12 \delta_1$ cannot be distinguished from $(\frac 12 + \ep) \delta_{-1} + (\frac 12 - \ep) \delta_1$ on the basis of $n$ samples when $\ep \asymp n^{-1/2}$.
The interesting part of \cref{intro-lower} is the rate when $d > 2p$.
We prove the following.
\begin{theorem}\label{dist-est-lb-general}
Let $d > 2p \geq 2$ and assume $\cX$ satisfies~\eqref{covering-assumption}.
Then
\begin{equation*}
R(n, \cP) \geq C_{d, p} (n \log n)^{-1/d}\,.
\end{equation*}
\end{theorem}
Before proving \cref{dist-est-lb-general}, we establish the two propositions described above.
\Cref{random_injection} allows us to reduce \cref{dist-est-lb-general} to an estimation problem involving total variation distance, and \cref{tv_lb} is a lower bound on the minimax rate for that total variation estimation problem.

\begin{proposition}\label{random_injection}
Assume $d > 2p \geq 2$, and let $m$ be a positive integer.
Let $u$ be the uniform distribution on~$[m]:=\{1, \ldots, m\}$.
There exists a random function~$F: [m] \to X$ such that for any distribution~$q$ on~$[m]$,
\begin{equation*}
c m^{-1/d} \tv{q}{u}^{\frac 1 p} \leq W_p(F_\sharp q, F_\sharp u) \leq C_{d, p} m^{-1/d} (\chi^2(q, u))^{1/d} \tv{q}{u}^{\frac 1 p - \frac 2 d}
\end{equation*}
with probability at least $.9$.
\end{proposition}
\begin{proof}
\Cref{packing-covering-duality} shows that the condition $\cN(\cX, \ep) \geq c \ep^{-d}$ implies the existence a set $\cG_m \defeq \{x_1, \dots, x_m\} \subseteq \cX$ such that $d(x_i, x_j) \gtrsim m^{-1/d}$ for all $i \neq j$.
We select $F$ uniformly at random from the set of all bijections from $[m]$ to $\cG_m$.

To show the lower bound, we note that any points $x, y \in \cG_m$ satisfy
\begin{equation*}
d(x, y)^p \gtrsim m^{-p/d} \1\{x \neq y\}\,,
\end{equation*}
which implies that for any coupling $\pi$ between $F_\sharp q$ and $F_\sharp u$
\begin{align*}
\int d(x, y)^p \dd \pi(x, y)& \gtrsim m^{-p/d} \p_\pi[X \neq Y]\\
& \geq m^{-p/d} \tv{F_\sharp q}{F_\sharp u} \\
&= m^{-p/d}\tv{q}{u}\,.
\end{align*}
The lower bound therefore holds with probability $1$.

We now turn to the upper bound.
We employ a dyadic covering bound~\cite[Proposition 1]{WeeBac18}.
For any $k^*$ there exists a dyadic partition $\{\cQ^k\}_{1 \leq k \leq k^*}$ of $X$ with parameter $\delta = 1/3$ such that $|\cQ^k| \leq \cN(\cX, 3^{-(k+1)})$.
We obtain that for any $k^* \geq 0$,
\begin{equation*}
W_p^p(F_\sharp q, F_\sharp u) \leq 3^{-k^* p} + \sum_{k=1}^{k^*} 3^{-(k-1)p} \sum_{Q_i^k \in \cQ^k} |F_\sharp q(Q_i^k) - F_\sharp u(Q_i^k)|\,,
\end{equation*}

By \cref{random_partition}, for any $k$,
\begin{equation*}
\E \sum_{Q_i^k \in \cQ^k} |F_\sharp q(Q_i^k) - F_\sharp u(Q_i^k)| \leq 2\tv{q}{u} \wedge C_{d, p}\left(\frac{3^{kd}\chi^2(q, u)}{m}\right)^{1/2} \,.
\end{equation*}
Let $k_0$ be a positive integer to be fixed later.

By applying the first bound, we obtain
\begin{equation*}
\E \sum_{k > k_0} 3^{-(k-1)p} \sum_{Q_i^k \in \cQ^k} |F_\sharp q(Q_i^k) - F_\sharp u(Q_i^k)| \lesssim 3^{-k_0 p} \tv{q}{u}\,.
\end{equation*}
Applying the second bound and recalling that $d/2 > p$ yields
\begin{align*}
\E \sum_{k \leq k_0} 3^{-(k-1)p} \sum_{Q_i^k \in \cQ^k} |F_\sharp q(Q_i^k) - F_\sharp u(Q_i^k)| & \leq C_{d, p} \left(\frac{\chi^2(q, u)}{m}\right)^{1/2} \sum_{k \leq k_0} 3^{k (d/2 - p)} \\
&  \leq C_{d, p} \left(\frac{\chi^2(q, u)}{m}\right)^{1/2} 3^{k_0(d/2 - p)}\,,
\end{align*}

We obtain for any $k^* \geq 0$ that 
\begin{equation*}
\E W_p^p(F_\sharp q, F_\sharp u) \leq C_{d, p} \left(\frac{\chi^2(q, u)}{m}\right)^{1/2} 3^{k_0(d/2 - p)} + C \cdot 3^{-k_0 p} \tv{q}{u} + 3^{-k^* p}\,,
\end{equation*}
and taking $k^* \to \infty$ it suffices to bound the first two terms.

Let $k_0$ to be the smallest positive integer such that
\begin{equation*}
3^{k_0 d} \geq m \frac{\tv{q}{u}^{2}}{\chi^{2}(q, u)}\,.
\end{equation*}
Then
\begin{equation*}
3^{k_0 d/2} \left(\frac{\chi^2(q, u)}{m}\right)^{1/2} \leq C_d \cdot \tv{q}{u}\,,
\end{equation*}
and hence
\begin{equation*}
\E W_p^p(F_\sharp q, F_\sharp u)\leq C_{d, p} \cdot 3^{-k_0 p} \tv{q}{u} \leq C_{d, p} m^{-p/d} (\chi^2(q, u))^{p/d} \tv{q}{u}^{1 - \frac{2p}{d}}\,.
\end{equation*}
The claim follows from Markov's inequality.
\end{proof}

We now show that there are composite hypotheses that are well separated in total variation distance but nevertheless hard to distinguish on the basis of samples.
\begin{proposition}\label{tv_lb}
Fix a positive integer $n$ and a constant $\delta \in [0, 1/10]$.
Given a positive integer $m$, let $\cD_{m}$ be the set of probability distributions $q$ on $[m]$ satisfying $\chi^2(q, u) \leq 9$.
Denote by $\cD_{m,\delta}^-$ the subset of $\cD_m$ of distributions satisfying $\tv{q}{u} \leq \delta$ and by $\cD_m^+$ the subset of $\cD_m$ satisfying $\tv{q}{u} \geq 1/4$.
If $m = \lceil C \delta^{-1} n \log n \rceil$ for a sufficiently large universal constant $C$ and $n$ is sufficiently large, then
\begin{equation*}
\inf_\psi \big\{\sup_{q \in \cD_{m}^+} \p_q[\psi =1] + \sup_{q \in \cD_{m,\delta}^-} \p_q[\psi =0] \big\} \geq  .9\,,
\end{equation*}
where the infimum is taken over all (possibly randomized) tests based on $n$ samples.
\end{proposition}

The proof of \cref{tv_lb} follows a strategy due to \citet{ValVal10} and \citet{WuYan19}, and our argument is a modification of theirs which permits simultaneous control of total variation and the $\chi^2$ divergence.
We give the proof in \cref{tv_lb_proof}.

We now give a proof of the main theorem.
\begin{proof}[Proof of \cref{dist-est-lb-general}]
Let $\hat W$ be any estimator for the Wasserstein distance between distributions on $\cX$ constructed on the basis of $n$ samples from each distribution.

Let $u$ be the uniform distribution on $[m]$, for some $m$ to be specified.
Let $c^*$ be the constant appearing in the lower bound of \cref{random_injection} and define $\Delta_d = \frac 1 {16} c^* m^{-1/d}$.
Given $n$ samples~$X_1, \dots, X_n$ from an unknown distribution on~$[m]$, define the randomized test
\begin{equation*}
\psi = \psi(X_1, \dots, X_n) \defeq \1\{\hat W(F(X_1), \dots, F(X_n); F(Y_1), \dots, F(Y_n)) \leq 2 \Delta_d\}\,,
\end{equation*}
where $F$ is the random function constructed in \cref{random_injection} and where $Y_i$ are i.i.d.\ from $u$.

By \cref{random_injection}, if $\delta \leq \delta_{d, p} \defeq \left(\frac{c^*}{176 C_{d, p}}\right)^{\frac{1}{1/p-2/d}}$, any $q \in \cD_{m, \delta}^-$ satisfies the bound $W_p(F_\sharp q, F_\sharp u) \leq \Delta_d$ with probability at least $.9$.
Likewise, for $q \in \cD_{m}^+$, the bound $W_p(F_\sharp q, F_\sharp u) \geq 3  \Delta_d$ also holds with probability at least $.9$.

Define the event $A=\{|\hat W - W_p(F_\sharp q, F_\sharp u)|  \geq \Delta_d\}$.
We obtain, for any $q \in \cD_{m, \delta}^-$,
\begin{align*}
\E_F \p_{F_\sharp q, F_\sharp u} [A] & \geq \E_F \p_{F_\sharp q F_\sharp u} [\hat W > 2 \Delta_d \text{ and } W_p(F_\sharp q, F_\sharp u) \leq \Delta_d] \\
& \geq \E_F \p_{F_\sharp q F_\sharp u} [\hat W > 2\Delta_d] - \p[W_p(F_\sharp q, F_\sharp u) > \Delta_d] \\\
& \geq \p_{q} [\psi = 0] - .1\,,
\end{align*}
and analogously for $q \in \cD_{m}^+$,
\begin{equation*}
\E_F \p_{F_\sharp q, F_\sharp u} [A] \geq \p_q [\psi = 1] - .1\,.
\end{equation*}

For any estimator $\hat W$, we have
\begin{align*}
\sup_{\mu, \nu \in \cP}
\p_{\mu, \nu}[
|\hat W - W_p(\mu, \nu)| \geq \Delta_d
]
& \geq
\frac 1 2
\big
(
\sup_{q \in \cD_m^+} \E_F \p_{F_\sharp q, F_\sharp u} [A]+
 \sup_{q \in \cD_{m,\delta}^-} \E_F \p_{F_\sharp q, F_\sharp u}[A]
\big
) \\
&  \geq \frac 12 \big(\sup_{q \in \cD_m^+} \p_q[\psi = 1] + \sup_{q \in \cD_{m,\delta}^-} \p_q[\psi = 0]\big) - .1\,.
\end{align*}

Choosing $m = \lceil C \delta^{-1} n \log n \rceil$ for a sufficiently large constant $C$ and applying \cref{tv_lb} yields that $\sup_{\mu, \nu \in \cP}
\p_{\mu, \nu}[|\hat W - W_p(F_\sharp q, F_\sharp u)|  \geq \Delta_d] \geq .8$, and Markov's inequality yields the claim.
\end{proof}

\section{Computational-statistical gaps for the spiked transport model}\label{sec:comp}
\Cref{sec:lower,sec:low-dim} clarify the statistical price for estimating the Wasserstein distance for high-dimensional measures.
\Cref{sec:low-dim} shows that the curse of dimensionality can be avoided under the spiked transport model.
The WPP estimator exploits the low-dimensional structure in the spiked transport model, thereby beating the worst-case rate presented in \cref{sec:lower}.
However, it is not clear how to make the estimator we propose computationally efficient.
In this section, we give evidence that this obstruction is a fundamental obstacle, that is, that no computationally efficient estimator can beat the curse of dimensionality.

The statistical query model, first introduced in the context of PAC learning~\citep{Kea98}, is a well known computational framework for analyzing statistical algorithms.
Instead of being given access to data points from a distribution, a statistical query (SQ) algorithm can approximately evaluate the expectation of arbitrary functions with respect to the distribution.
This model naturally captures the power of noise-tolerant algorithms~\citep{Kea98} and is strong enough to implement nearly all common machine learning procedures~\citep[see, e.g.][]{BluDwoMcS05}.

We recall the following definition.
\begin{defin}
Given a distribution $\cD$ on $\RR^d$, for any sample size parameter~$t > 0$ and function $f: \RR^d \to [0, 1]$, the oracle $\vstat(t)$ returns a value $v \in [p - \tau, p + \tau]$, where $p = \E f(X)$ and $\tau = \frac 1 t \vee \sqrt{\frac{p(1-p)}{t}}$.
\end{defin}

A query to a $\vstat(t)$ oracle can be simulated by using a data set of size approximately $t$.
Our main result proves a lower bound against an oracle with sample size parameter $t = 2^{cd}$ for a positive constant $c$.
Simulating such an oracle would require a number of samples exponential in the dimension.
Nevertheless, we show that even under this strong assumption, at least $2^{cd}$ queries to the oracle are required.
This result suggests that any \emph{computationally efficient} procedure to estimate the Wasserstein distance under the spiked transport model requires an exponential number of samples.
By contrast, \cref{sec:upper} establishes that, information theoretically, only a polynomial number of samples are required.

We now state our main result.
\begin{theorem}\label{sq_lb}
There exists a positive universal constant $c$ such that, for any~$d$, estimating $W_1(\muone, \mutwo)$ for distributions $\muone$ and $\mutwo$ on $\RR^d$ satisfying the spiked transport assumption with $k = 1$ to accuracy $\Theta(1/\sqrt d)$ with probability at least $2/3$ requires at least $2^{cd}$ queries to $\vstat(2^{cd})$.
\end{theorem}

Our proof is based on a construction due to~\cite{DiaKanSte17} \citep[see also][]{BubLeePri19}.
We defer the details to the appendix.

\appendix
\section{Proofs of Lower Bounds}
\subsection{Proof of \cref{deviation_lb}}\label{deviation_lb_proof}
We reduce from the spiked covariance model.
By homogeneity, we may assume that $\sigma = 1$.
Let $\mu^{(1)}$ be the standard Gaussian measure on $\RR^d$, and for $\mu^{(2)}$ we take either the standard Gaussian measure or the distribution of a centered Gaussian with covariance $I + \beta v v^\top$, where $\|v\| = 1$ and $\beta > 0$ is to be specified.
As long as $\beta \lesssim 1$, the measure $\mu^{(2)}$ is a $O(1)$-Lipschitz pushforward of the Gaussian measure.
Hence, it satisfies $T_p(O(1))$~\citep[Corollary 20]{Goz07}.

Note that if $\mu^{(2)}$ has covariance $I + \beta v v^\top$, then
\begin{equation*}
W_p(\muone, \mutwo) \geq W_1(\muone, \mutwo) = W_1(\cN(0, 1), \cN(0, 1+\beta)) \gtrsim \beta\,.
\end{equation*}
However, \citet[Proposition 2]{CaiMaWu15} establish that the minimax testing error for
\begin{equation*}
\mathrm{H}_0: \cN(0, I) \text{ vs. } \mathrm{H}_1: \cN(0, I + \beta v v^\top), \|v\| = 1
\end{equation*}
is bounded below by a constant when $\beta \lesssim \sqrt{d/n}$.
A standard application of Le Cam's two-point method~\citep{Tsy09} yields the claim.

\subsection{Proof of \cref{tv_lb}}\label{tv_lb_proof}

We require the existence of two distributions on $\RR_+$, which will serve as the building blocks of our construction.

\begin{proposition}\label{priors}
For any integer $L \geq 0$ and $\ep \in [0, 1/6]$, there exists a pair of random variables $U$ and $V$ with the following properties:
\begin{itemize}
\item $\E U^j = \E V^j \quad \forall j \leq L$
\item $U, V \in [0, 16\ep^{-1} L^2]$ almost surely
\item $\E U = \E V = 1$ and $\E U^2 = \E V^2 \leq 6$.
\item $\E |U - 1| \leq 12 \ep$ but $\E |V - 1| \geq 1$.
\end{itemize}
\end{proposition}
The proof is deferred to \cref{sec:priors_proof}.

\begin{proof}[Proof of \cref{tv_lb}]
The proof follows closely the approach of \citet[Proposition 1]{WuYan19}.
We first employ a standard argument showing that we can consider the \emph{Poissonalized} setting.
It is trivial to see that given samples $X_1, \dots, X_n$ from a distribution $q$ on~$[m]$, the \emph{counts} $N_i = N_i(X_1, \dots, X_n) \defeq |\{j \in [n]: X_j = i\}|$ are sufficient for $q$.
We therefore consider tests $\psi$ based on count vectors.
Note that, under $q$, the count vector $(N_1, \dots, N_m)$ has distribution $\multi(n, q)$.

Define
\begin{equation*}
R_n \defeq \inf_\psi \big\{\sup_{q \in \cD_{m}^+} \p_q[\psi =1 ] + \sup_{q \in \cD_{m,\delta}^-} \p_q[\psi =0 ]\}\,.
\end{equation*}
We aim to prove a lower bound on $R_n$.

Let $\rho > 0$, and let for $n \geq 1$ let $\{\psi_n\}$ be a set of near optimal tests for a fixed sample size; i.e.
\begin{equation*}
\sup_{q \in \cD_{m,\delta}^- \cup \cD_m^+} \p_q[\psi_n \neq \1\{q \in \cD_{m, \delta}^-\}] \leq R_n + \rho\,.
\end{equation*}
Define set of approximate probability vectors
\begin{align*}
\tilde \cD_{m,\delta}^- &\defeq \left\{q \in \RR^m_+ : \left|\sum_{i=1}^m q_i - 1\right| \leq \delta , \frac{q}{\sum_{i=1}^m q_i} \in \cD_{m,\delta}^-\right\} \\
\tilde \cD_{m}^+ & \defeq \left\{q \in \RR^m_+ : \left|\sum_{i=1}^m q_i - 1\right| \leq \delta , \frac{q}{\sum_{i=1}^m q_i} \in \cD_m^+\right\}\,.
\end{align*}

We let $\tilde \cD_{m, \delta} \defeq \tilde \cD_{m,\delta}^- \cup \tilde \cD_{m}^+$.
Given $q \in \tilde \cD_{m, \delta}$, define the renormalization $\bar q=\sum_{i=1}^m q_i/q$.
We then define
\begin{equation*}
\tilde R_n \defeq \inf_\psi \big\{\sup_{q \in \tilde \cD_{m}^+} \p_q[\psi =1] + \sup_{q \in \tilde \cD_{m,\delta}^-} \p_q[\psi =0]\big\}\,,
\end{equation*}
where the infimum is taken over all estimators based on the counts $N_1, \dots, N_m$ and where $\p_q$ indicates the probability when $N_1, \dots, N_m$ are independent and $N_i \sim \poi(n q_i)$ for all $i \in [m]$.
We set $N = \sum_{i=1}^m N_i$, and note that, conditioned on $N = n'$, the count vector $(N_1, \dots, N_m)$ has distribution $\multi(n', \bar q)$.

We define a test $\tilde \psi$ based on these Poissonalized counts by setting
\begin{equation*}
\tilde \psi(N_1, \dots, N_m) \defeq \psi_N(N_1, \dots, N_m)\,.
\end{equation*}
This definition along with the near optimality of $\psi_{n'}$ for $n' \geq 0$ implies
\begin{align*}
\sup_{q \in \tilde \cD_{m}^+} \p_q[\tilde \psi =1] + \sup_{q \in \tilde \cD_{m,\delta}^-} \p_q[\tilde \psi =0] & \leq \sum_{n' \geq 0} R_{n'} \p_q[N = n'] + \rho\\
& \leq R_{n/2} + \p_q[N < n/2] + \rho\,,
\end{align*}
where the last inequality follows from the fact that $R_{n'} \leq 1$ for all $n' \geq 0$ and $R_{n'}$ is non-increasing in $n'$.
Since $N = \poi(n \sum_{i=1}^m q_i)$ and $\sum_{i=1}^m q_i \geq 3/4$, a standard Chernoff bound implies $\p[N < n/2] \leq \exp(-Cn)$.
Since $\rho$ was arbitrary, we obtain that
\begin{equation*}
\tilde R_n \leq R_{n/2} + \exp(-Cn)\,.
\end{equation*}

To prove a lower bound on $\tilde R_n$, we consider random vectors
\begin{align*}
Q & = \frac 1 m (U_1, \dots, U_m) \\
Q'& = \frac 1 m (V_1, \dots, V_m)\,,
\end{align*}
where $U_i$ and $V_i$ for $i \in [m]$ are independent copies of $U$ and $V$ constructed in \cref{priors} with $\ep = \frac{1}{24}\delta$.
Conditioned on $Q$ and $Q'$, let $N$ and $N'$ be count vectors with independent entries generated by $N_i \sim \poi(nQ_i)$ and $N'_i \sim \poi(nQ'_i)$.
Let us denote by $\mathrm{P}$ and $\mathrm{P'}$ the distributions of $N$ and $N'$ respectively.
Under $\mathrm{P}$ and $\mathrm{P'}$, the entries of $N$ and $N'$ are i.i.d.\ Poisson mixtures, so applying~\citet[Lemma 4]{WuYan19} yields
\begin{equation*}
\tv{\mathrm{P}}{\mathrm{P'}} \leq m \left(\frac{8 e n L}{\ep m}\right)^L \,.%
\end{equation*}
Let $E = \{Q \in \tilde \cD_{m, \delta}^-\}$ and $E' =\{Q' \in \tilde \cD_{m}^+\}$.
By \cref{approximate-probabilities}, $\p[E^C]$ and $\p[{E'}^C]$ are each at most $C \frac{L^4}{\delta^2 m}$.

Let $\pi_E$ be the law of $Q$ conditioned on $E$, and define $\pi'_{E'}$ analogously, and let $\mathrm{P}_E$ and $\mathrm{P}'_{E'}$ be the laws of $N$ and $N'$ under these priors.
We obtain for any estimator $\psi$ based on count vectors
\begin{align*}
\sup_{q \in \tilde \cD_{m}^+} \p_q[\tilde \psi =1] + \sup_{q \in \tilde \cD_{m,\delta}^-} \p_q[\tilde \psi =0] & \geq  \int \p_{q'}[\psi = 1] \dd \pi'_{E'}(q') + \int \p_q[\psi =0] \dd \pi_E(q) \\
& \geq 1 - \tv{\mathrm{P}_E}{\mathrm{P}'_{E'}} \\
& \geq 1 - \tv{\mathrm{P}}{\mathrm{P'}} - C \frac{L^4}{\delta^2 m}\,.
\end{align*}
Choosing $L = c\frac{\delta m}{n}$ for a sufficiently small constant $c$ yields that
\begin{equation*}
\tilde R_n \geq 1-  m \exp(- C \frac{\delta m}{n}) - C \frac{\delta^2 m^3}{n^4}\,.
\end{equation*}
Therefore
\begin{equation*}
R_n \geq 1 - m \exp(- C \frac{\delta m}{n}) - C \frac{\delta^2  m^3}{n^4} - \exp(-C n)\,,
\end{equation*}
and choosing $m = \lceil C \delta^{-1} n \log n \rceil$ for $C$ a sufficiently large constant and $n$ sufficiently large yields the claim.
\end{proof}

\subsection{Proof of Proposition~\ref{priors}}\label{sec:priors_proof}
First, the reduction of \citet[Lemma 7]{WuYan19} implies that it suffices to construct random variables $Y$ and $Y'$ such that
\begin{itemize}
\item $\E Y^j = \E {Y'}^j \quad \forall 0 \leq j < L$
\item $Y, Y' \in [1, 16 \ep^{-1} L^2]$ a.s.
\item $\E Y = \E Y' \leq 6$
\item $\E \frac 1 Y \geq 1 - 6 \ep$ but $\E \frac{1}{Y'} \leq \frac 1 2$.
\end{itemize}
Indeed, applying their construction yields $U$ and $V$ satisfying the first three requirements of \cref{priors} as well as $\p[U = 0] \leq 6 \ep$ and $\p[V = 0] \geq \frac 1 4$.
Since the supports of $U$ and $V$ lie in $\{0\} \cup [1, +\infty)$, we have $\E|U - 1| = 2 \p[U = 0] \leq 12 \ep$ and $\E|V - 1| = 2 \p[V = 0] \geq 1$, as desired.
We therefore focus on constructing such a $Y$ and $Y'$.

By \citet[Lemma 7]{WuYan19} combined with \citet[Section 2.11.1]{Tim94}, there exist random variables $X$ and $X'$ supported on $[1, 16 L^2]$ such that $\E X^j = \E {X'}^j$ for  $0 \leq j < L$ and $\E \frac 1 X - \E \frac{1}{X'} \geq \frac 1 2$.
Let $\mathrm{P}_\ep$ and $\mathrm{P}'_\ep$ denote the distribution of $\ep^{-1} X$ and $\ep^{-1} X'$, respectively.

Let
\begin{align*}
\Delta_\ep &\defeq \int \frac{1}{(y-1)(y-2)} \dd \mathrm{P}_\ep (y)\\
\Delta'_\ep &\defeq \int \frac{1}{(y'-1)(y'-2)} \dd \mathrm{P}'_\ep (y')\\
\cZ_\ep & \defeq \int \frac{1}{y - 2} \dd \mathrm{P}_\ep(y) - \int \frac{1}{y' - 1} \dd \mathrm{P}'_\ep (y')\,.
\end{align*}
We define two new distributions $\mathrm{Q}$ and $\mathrm{Q'}$ by
\begin{align}
\mathrm{Q}(\mathrm{d}y) & = \delta_1(\mathrm{d}y) + \frac{1}{\cZ_\ep}\left(\frac{1}{(y-1)(y - 2)} \mathrm{P}_\ep(\mathrm{d}y) - \Delta_\ep \delta_1(\mathrm{d}y)\right) \label{eq:py}\\
\mathrm{Q'}(\mathrm{d}y') & = \delta_2(\mathrm{d}y') + \frac{1}{\cZ_\ep}\left(\frac{1}{(y'-1)(y' - 2)} \mathrm{P}'_\ep(\mathrm{d}y') - \Delta_{\ep'} \delta_2(\mathrm{d}y')\right)\label{eq:py-prime}
\end{align}

By \cref{rescaled-prior-bounds},
\begin{align*}
\Delta_\ep, \Delta_{\ep'} & \in [0, \frac{9}{5} \ep^2] \\
\cZ_\ep & \geq \frac{3}{10} \ep\,,
\end{align*}
which implies in particular that both $Q$ and $Q'$ are probability distributions.

Let $Y \sim Q$ and $Y' \sim Q'$.
We first check the last three conditions.
Clearly $Y$ and $Y'$ are supported on $[1, 16 \ep^{-1} L^2]$,
and \cref{prior-first-moments-match} implies that $\E Y = \E Y' \leq 6$.
We have $\E \frac 1 Y \geq \p[Y = 1] = 1 - \frac{\Delta_\ep}{\cZ_\ep} \geq 1 - 6 \ep$, and since $Y' \geq 2$ almost surely the bound $\E \frac{1}{Y'} \leq \frac 1 2$ is immediate.

It remains to check the moment-matching condition.
Any polynomial $p(y)$ of degree at most $L-1$ can be written 
\begin{equation*}
p(y) = (y-1)(y-2)q(y) + \alpha y + \beta\,,
\end{equation*}
where $q(y)$ has degree less than $L-1$.
Then
\begin{equation*}
\E p(Y) - \E p(Y') = \E (Y-1)(Y-2)q(Y) - \E (Y'-1)(Y'-2)q(Y') + \alpha (\E Y - \E Y')\,.
\end{equation*}
The last term vanishes because $\E Y = \E Y'$, and 
\begin{equation*}
\E (Y-1)(Y-2)q(Y) - \E (Y'-1)(Y'-2)q(Y') = \frac{1}{\cZ_\ep} \left(\E q(\ep^{-1} X) - \E q(\ep^{-1} X')\right) = 0\,,
\end{equation*}
since $\E X^j = \E {X'}^j$ for all $j < L - 1$.

Therefore $\E p(Y) = \E p(Y')$ for all polynomials of degree at most $L - 1$.
\qed

\subsection{Proof of \cref{sq_lb}}
We first establish the existence of a probability distribution on $\RR$ which agrees with $\cN(0, 1)$ on many moments, but is far from $\cN(0, 1)$ in Wasserstein distance.

\begin{proposition}\label{one_dimensional_distribution}
There exists a $O(1)$-subgaussian distribution $A$ on $\RR$ that satisfies the following requirements.
\begin{itemize}
\item $A$ agrees with $\cN(0, 1)$ on the first $2m - 1$ moments.
\item $W_1(A, \cN(0, 1)) = \Omega(1/\sqrt m)$.
\item $\chis{A}{\cN(0, 1)} = \exp(O(m))$.
\end{itemize}
\end{proposition}
\begin{proof}
By~\citet[Corollary 4.4]{DiaKanSte17}, for any $\delta \in (0, 1)$ we can find an atomic measure $Q$ supported on $m$ points in $[-O(\sqrt m), O(\sqrt m)]$ such that $A \defeq Q * \cN(0, \delta)$ matches $\cN(0, 1)$ on the first $2m - 1$ moments.
By~\citet[Lemma 4.6]{DiaKanSte17}, this distribution satisfies the bound $\chis{A}{\cN(0, 1)} = \exp(O(m))/\sqrt \delta$.
Moreover, the distribution $Q$ is supported on the zeros of a rescaled version of the $m$th Hermite polynomial, and by \citet[Lemma A.1]{BubLeePri19} these zeros are $\Omega(1/\sqrt m)$ apart.

Since the points in the support of $Q$ are $\Omega(1/\sqrt m)$ apart, there exists a constant $c$ such that the union of balls of radius $c/\sqrt m$ centered at the support of $Q$ covers at most half of the interval $[-1, 1]$, and since the Gaussian density is bounded below on this interval, a constant fraction of the mass of $\cN(0, 1)$ is located at distance at least $c/\sqrt m$ from $Q$.
Hence $W_1(Q, \cN(0, 1)) = \Omega(1/\sqrt m)$.
Clearly $W_1(A, Q) = O(\sqrt \delta)$.

Therefore, if we choose $\delta = O(1/m)$, then
\begin{equation*}
W_1(A, \cN(0, 1)) \geq W_1(Q, \cN(0, 1)) - W_1(A, Q) = \Omega(1/\sqrt m)
\end{equation*}
and $\chis{A}{\cN(0, 1)} = O(\sqrt m) \exp(O(m)) = \exp(O(m))$.

Finally, we show that $A$ is $O(1)$-subgaussian, and therefore satisfies $\tone{C}$ for a positive constant $C$.
Standard facts~\citep[see][]{Ver18} imply that it suffices to show that if $X \sim A$, then $\|X\|_k = O(\sqrt k)$ for all $k$.
This clearly holds for $k \leq 2m -1$, since $\cN(0, 1)$ is itself $1$-subgaussian and the first $2m -1$ moments of $A$ and $\cN(0, 1)$ agree.
On the other hand, for $k \geq 2m$, if $Y \sim Q$ and $Z \sim \cN(0, 1)$, then $X = Y + \sqrt \delta Z \sim A$, and
\begin{equation*}
\|X + \sqrt \delta Z\|_k \leq \|X\|_\infty + \sqrt \delta \|Z\|_k \lesssim \sqrt m + \sqrt{k/m} \lesssim \sqrt k\,,
\end{equation*}
as desired.
\end{proof}

The separation $W_1(A, \cN(0, 1)) = \Omega(1/\sqrt m)$ in  \cref{one_dimensional_distribution} is easily seen to be tight.
Indeed, \citet[Corollary 2]{RigWee19} show that if $\mu$ and $\nu$ are $O(1)$-subgaussian and agree on their first $O(m)$ moments, then $W_1(\mu, \nu) = O(1/\sqrt m)$.

By planting the distribution constructed in \cref{one_dimensional_distribution} in a random direction, we obtain two high-dimensional measures satisfying the spiked transport model.
\begin{lemma}\label{A_spike}
Let $v$ be a unit vector in $\RR^d$, and denote by $P_v$ the distribution on $\RR^d$ of the random variable $X v + Z$, where $X \sim A$ and $Z \sim \cN(0, I_d - v v^\top)$ is independent of $X$.
Then $\mu^{(1)}=P_v$ and $\mu{(2)}=\cN(0, I_d)$ satisfy the spiked transport model~\eqref{eq:spiked_transport}, and $W_1(P_v, \cN(0, I_d)) = \Omega(1/\sqrt m)$.
\end{lemma}
\begin{proof}
If we let $\xi \sim \cN(0, 1)$, then $\cN(0, I_d)$ is the law of $\xi v + Z$, where $\xi$ and $Z$ are independent and $Z \sim \cN(0, I_d - v v^\top)$.
Denoting by $\cU$ the span of $v$, we see that $P_v$ and $\cN(0, I_d)$ satisfy~\eqref{eq:spiked_transport} with $X^{(1)} = X v$ and $X^{(2)} = \xi v$.
By \cref{prop:wpk-is-wp,one_dimensional_distribution}, $W_1(P_v, \cN(0, I_d)) = W_1(A, \cN(0, 1)) = \Omega(1/\sqrt m)$.
\end{proof}

The proof of \cref{sq_lb} follows from a framework due to~\citet{FelGriRey17}, from which the following result is extracted.
Given distributions $P_1$, $P_2$, and $Q$, define
\begin{equation*}
\chi^2_{Q}(P_1, P_2) \defeq \int \left(\frac{dP_1}{dQ} - 1\right)\left(\frac{dP_2}{dQ} - 1\right) \, \dd Q\,.
\end{equation*}
We call a set $\cP$ of distributions $(\gamma, \beta)$ correlated with respect to $Q$ if for all $P_i, P_j \in \cP$,
\begin{equation*}
\chi^2_Q(P_i, P_j) \leq \left\{\begin{array}{ll}
\beta & \text{if $i = j$,} \\
\gamma & \text{if $i \neq j$.}
\end{array}\right.
\end{equation*}
We then have the following.
\begin{proposition}\label{decision_lb}
Let $\cQ$ be a set of distributions, and let $Q$ be a reference distribution.
Suppose that there exists a set $\cP \subseteq \cQ$ such that $\cP$ is $(\gamma, \beta)$ correlated with respect to $Q$.
Then any SQ algorithm that distinguishes queries from $P = Q$ and $P \in \cQ$ requires at least $|\cP| \gamma/3\beta$ queries to $\vstat(1/2 \gamma)$
\end{proposition}
\begin{proof}
By choosing $\gamma' = \gamma$ in~\citet[Lemma 3.10]{FelGriRey17}, we obtain that the set $\cP$ satisfies $\mathrm{SDA}(\cP, Q, 2\gamma) \geq |\cP|\gamma/(\beta - \gamma) \geq |\cP| \gamma/\beta$.
Then \citet[Theorem 3.7]{FelGriRey17} implies that distinguishing $Q$ from $\cQ$ requires at least $|\cP| \gamma/3\beta$ queries to $\vstat(1/2 \gamma)$.
\end{proof}

We can now prove the lower bound.
\begin{proof}[Proof of \cref{sq_lb}]
Let $\cQ$ be the set $\{P_v : v \in \RR^d, \|v\| = 1\}$.
The Johnson-Lindenstrauss lemma~\citep{JohLin84} implies that for any $\delta \in (0, 1)$, there exists a set of $2^{\Omega(\delta^2 d)}$ unit vectors in $\RR^d$ with pairwise inner product at most $\delta$.
Denote by $\cS$ a set of such vectors, and set $\cP \defeq \{P_v : v \in \cS\} \subseteq \cQ$.

Write $\Delta$ for $\chis{A}{\cN(0, 1)}$, and recall that $\Delta = \exp(O(m))$.
By~\citet[Lemma 3.4]{DiaKanSte17}, $\chi^2_{\cN(0, I)}(P_v, P_{v'}) \leq |v \cdot v'|^{2m} \Delta$, and the set $\cP$ is therefore $(\delta^{2m} \Delta, \Delta)$ correlated.
By \cref{decision_lb}, any SQ algorithm that distinguishes between $\mu = \cN(0, 1)$ and $\mu \in \cP$ with probability at least $2/3$ requires $2^{\Omega(\delta^2 d)}\delta^{2m}$ queries to $\vstat(\delta^{-2m}/2 \Delta)$.

Let $\delta > 0$ be a constant small enough that $\delta^{-2m}/2\Delta = \exp(\Omega(m))$.
If $m = c d$ for a sufficiently small positive constant $c$, then $2^{\Omega(\delta^2 d)} \delta^{2m} = \exp(\Omega(d))$.
An SQ algorithm to distinguish queries from $\cN(0, 1)$ from those from a distribution in $\cQ$ with probability at least $2/3$ therefore requires $\exp(\Omega(d))$ queries to $\vstat(\exp(\Omega(d))$.
Therefore, by \cref{A_spike}, any SQ algorithm which estimates $W_1$ under the spiked transport model to accuracy $\Theta(1/\sqrt m) = \Theta(1/\sqrt d)$ requires $\exp(\Omega(d))$ queries to $\vstat(\exp(\Omega(d))$, as claimed.
\end{proof}

\section{Supplementary material}\label{sec:omitted}
\subsection{Additional lemmas}\label{subsec:lemmas}

\begin{lemma}\label{lem:vk_op_covering}
There exists a universal constant $c$ such that for all $\ep \in (0, 1)$, the covering number of $\cV_k(\RR^d)$ satisfies
\begin{equation*}
\log \cN(\cV_k, \ep, \op{\cdot}) \leq dk \log \frac{c\sqrt k} \ep\,.
\end{equation*}
\end{lemma}
\begin{proof}
If we denote by $\cS_d$ the unit sphere in $\RR^d$, then $\cV_k(\RR^d) \subseteq \cS_d^{\times k}$, and
\begin{equation*}
\op{U - V}^2 \leq \sum_{i = 1}^k \|U_i - V_i\|^2\,,
\end{equation*}
where $\{U_i\}$ (resp. $\{V_i\}$) are the rows of $U$ (resp. $V$).
The claim follows immediately from the fact that $\log \cN(\cS_d, \ep, \|\cdot\|) \leq d \log \frac c \ep$ for a universal constant $c$.
\end{proof}

\begin{lemma}\label{prop:subg_vec_norm}
Let $\mu$ be a centered distribution satisfying $\tone{\sigma^2}$.
If $X \sim \mu$, then $(\E \|X\|^p)^{1/p} \lesssim \sqrt{d p}$.
\end{lemma}
\begin{proof}
By the monotonicity of $L_p$ norms, we may assume that $p$ is a positive even integer. The claim then follows from \cref{prop:norm_subG} and standard bounds on the moments of subgaussian random variables~\cite{Ver18}.
\end{proof}

The following proposition provides a sort of converse to \cref{prop:wpk-is-wp}: if the $k$-dimensional Wasserstein distance agrees with the Wasserstein distance, then there is an optimal coupling between the measures which acts only on a $k$-dimensional subspace.

\begin{proposition}\label{prop:low-dim-coupling}
If $\wpk(\mu, \nu) = W_p(\mu, \nu)$, then there exists a coupling $\gamma$ optimal for $\mu$ and $\nu$ such that $X - Y$ is supported on a $k$-dimensional subspace.
\end{proposition}
\begin{proof}
Let $\gamma$ be an optimal coupling for $\mu$ and $\nu$, and write $$U = \argmax_{V \in \stie} W_p(\mu_V, \nu_V)\,.$$
Then
\begin{align*}
W_p^p(\mu, \nu)  =  \int \|x - y\|^p \dd\gamma(x, y)  \geq \int \|U (x - y)\|^p \dd \gamma(x, y)  \geq \wpk^p(\mu, \nu)\,,
\end{align*}
where the first inequality uses the fact that $\op{U} = 1$ and the second uses the fact that if $(X, Y) \sim \gamma$, then $U X$ and $U Y$ are distributed according to $\mu_U$ and $\nu_U$, respectively.
If the first equality holds, then we must have $\|X - Y\| = \|U(X - Y)\|$ $\gamma$-almost surely, and if the second inequality holds then this coupling is optimal for $\mu_U$ and $\nu_U$
\end{proof}

\begin{proposition}[{\citealp[Proposition~2.5.2]{Ver18}}]\label{prop:second_moment_to_subG}
If a real-valued random variable $X$ is $\sigma^2$-subgaussian, then
\begin{equation}\label{eq:second_moment}
\E e^{\lambda^2 X^2} \leq e^{4\lambda^2 \sigma^2} \quad \quad \forall |\lambda| \leq \frac{1}{2\sigma}\,.
\end{equation}

Conversely, if~\eqref{eq:second_moment} holds and $\E X = 0$, then $X$ is $8\sigma^2$-subgaussian.
\end{proposition}
\begin{proof}
For the first claim, let $Z$ be a standard Gaussian random variable independent of $X$. Then, if $\lambda^2 \sigma^2 < 1/4$,
\begin{equation*}
\E e^{\lambda^2 X^2} = \E e^{\sqrt 2 \lambda Z X} \leq \E e^{\lambda^2 \sigma^2 Z^2} = \frac{1}{(1 - 2 \lambda^2 \sigma^2)^{1/2}} \leq e^{4 \lambda^2 \sigma^2}\,,
\end{equation*}
where the last step uses the inequality $(1-x)^{-1} \leq e^{2x}$ for $0 \leq x \leq \frac 12$.

Conversely, suppose that~\eqref{eq:second_moment} holds and $\E X = 0$.
Then, for $|\lambda| \leq \frac{1}{2 \sigma}$, we have
\begin{equation*}
\E e^{\lambda X} \leq \E \lambda X + \E e^{\lambda^2 X^2} \leq e^{4 \lambda^2 \sigma^2}\,.
\end{equation*}
If $|\lambda| > \frac{1}{2 \sigma}$, then
\begin{equation*}
\E e^{\lambda X} \leq \E e^{2 \lambda^2 \sigma^2 + \frac{X^2}{8 \sigma^2}} \leq e^{2 \lambda^2 \sigma^2 + 1} \leq e^{4 \lambda^2 \sigma^2}\,.
\end{equation*}
\end{proof}

\begin{proposition}\label{prop:norm_subG}
If $X$ on $\RR^k$ is $\sigma^2$-subgaussian, then $\ep \|X\|$ is $(8 k \sigma^2)$-subgaussian, where $\ep$ is a Rademacher random variable independent of $X$.
\end{proposition}
\begin{proof}
For $|\lambda| < \frac{1}{2 \sqrt k\sigma}$,
\begin{align*}
\E \exp(\lambda^2 (\ep \|X\|)^2) & = \E \exp(\lambda^2 \|X\|^2)  \leq \E \exp(k \lambda^2 (e_i^\top X)^2)  \leq \exp(4 k \lambda^2 \sigma^2)\,.
\end{align*}
Applying \cref{prop:second_moment_to_subG} yields the claim.
\end{proof}

\begin{proposition}\label{subgaussian_wp_rate}
Let $p' \in [1, 2]$. If $\rho_n$ on $\RR^k$ satisfies $\tpp{\sigma^2}$, then for any $p \in [1, \infty)$,
\begin{equation*}
\E W_{p}\left(\rho_n, \rho\right) \leq \rate{n} \defeq c_{p} \sigma \sqrt k \left\{
\begin{array}{ll}
n^{-1/2p} & \text{if $k < 2p$} \\
n^{-1/2p} (\log n)^{1/p} & \text{if $k = 2p$} \\
n^{-1/k} & \text{if $k > 2p$}
\end{array}
\right.
\end{equation*}
\end{proposition}
\begin{proof}
\Cref{lem:t1_is_subg} implies that $X \sim \rho$ is $\sigma^2$ subgaussian.
in particular, $X$ satisfies $(\E \|X\|^q)^{1/q} \leq \sigma \sqrt{qk}$.
Choosing $q = 3p$ and using \citet[Theorem 3.1]{Lei18} implies the claim.
\end{proof}

\begin{lemma}\label{packing-covering-duality}
If $\cN_\ep(\cX) \geq c \ep^{-d}$, then for any positive integer $m$ there exists a subset of $\cX$ of cardinality $m$ such that each pair of points is separated by at least $\frac 1 2 \left(\frac c m\right)^{1/d}$.
\end{lemma}
\begin{proof}
Let $S$ be a maximal subset of $\cX$ such that each pair of points in $S$ is separated by at least $\frac 1 2 \left(\frac c m\right)^{1/d}$.
Then, for any point $x \in X$, there must be an $s \in S$ such that $d(x, s) < \frac 1 2 \left(\frac c m\right)^{1/d}$, since otherwise $x$ could be added to $S$.
Therefore, balls of radius $\frac 1 2 \left(\frac c m\right)^{1/d}$ around the points $s \in S$ cover $X$, which implies
\begin{equation*}
|S| \geq \cN(X, \left(\frac c m\right)^{1/d}) \geq m\,,
\end{equation*}
as claimed.
\end{proof}

\begin{lemma}\label{random_partition}
Let $Q_1, \dots, Q_\ell$ be a partition of $\cX$. If $F$ is a uniform random bijection from $[m]$ to $\cG_m$, then
\begin{equation*}
\E \sum_{i=1}^\ell \left| F_\sharp q(Q_i) - F_\sharp u (Q_i) \right| \leq \|q - u\|_1 \wedge \sqrt{\frac{\ell \cdot \chi^2(q, u)} m} 
\end{equation*}
\end{lemma}
\begin{proof}
The first bound follows immediately from the triangle inequality:
\begin{equation*}
\sum_{i=1}^\ell \left| F_\sharp q(Q_i) - F_\sharp u (Q_i) \right| \leq \sum_{i=1}^\ell \sum_{j: F(j) \in Q_i} \left|q(j) - \frac 1 m \right| = \|q - u\|_1\,.
\end{equation*}

We now show the second bound.
Fix $i \in [\ell]$.
Under the random choice of $F$, the quantity $F_\sharp q(Q_i) - F_\sharp u (Q_i) = \sum_{j: F(j) \in Q_i} \left(q(j) - \frac{1}{m}\right)$ is a sum of $|Q_i|$ terms selected uniformly \emph{without replacement} from the set $\Delta \defeq \{q(j) - \frac 1 m: j \in [m]\}$.
By \citet[Theorem 4]{Hoe63}, since $x \mapsto |x|$ is continuous and convex,
\begin{equation*}
\E \left| \sum_{j: F(j) \in Q_i} \left(q(j) - \frac{1}{m}\right)\right| \leq \E \left| \sum_{j=1}^{|Q_i|} \Delta_j \right|\,,
\end{equation*}
where $\Delta_j$ are selected uniformly \emph{with replacement} from $\Delta$.
Note that $\E \Delta_j = 0$, and $\E \Delta_j^2 = \frac 1 m \sum_{j=1}^m (q(j) - \frac 1 m )^2 = m^{-2} \chi^2(q, u)$.
Jensen's inequality then yields
\begin{equation*}
\E \left| \sum_{j=1}^{|Q_i|} \Delta_j \right| \leq \frac{\sqrt{|Q_i|\chi^2(q, u)} }{m} \,.
\end{equation*}
The Cauchy-Schwarz inequality finally implies
\begin{equation*}
\E \sum_{i=1}^\ell \left| \sum_{j: F(j) \in Q_i} \left(q(j) - \frac{1}{m}\right)\right| \leq \sum_{i=1}^\ell\frac{\sqrt{|Q_i|\chi^2(q, u)}}{m}  \leq \sqrt{\frac {\ell \cdot \chi^2(q, u)} m} \,.
\end{equation*}
\end{proof}

\begin{lemma}\label{rescaled-prior-bounds}
Let $\ep \in [0, 1/6]$.
If $X, X' \geq 1$ almost surely and $\E \frac 1 X - \E \frac{1}{X'} \geq \frac 12$, then
\begin{equation*}
\E \frac{1}{(\ep^{-1} X - 1)(\ep^{-1} X - 2)}, \E \frac{1}{(\ep^{-1} X' - 1)(\ep^{-1} X' - 2)} \in [0, \frac 9 5 \ep^2]
\end{equation*}
and
\begin{equation*}
\E \frac{1}{\ep^{-1} X - 2} - \E \frac{1}{\ep^{-1} X' - 1} \geq \frac{3}{10}\ep\,.
\end{equation*}
\end{lemma}
\begin{proof}
If $x \geq 1$ and $\ep \leq 1/6$, then $(x - \ep)(x - 2\ep) \geq \frac{5}{9}$, so
\begin{equation*}
\frac{1}{(\ep^{-1} x - 1)(\ep^{-1} x - 2)} = \frac{\ep^2}{(x - \ep)(x - 2\ep)} \leq  \frac 9 5 \ep^2\,,
\end{equation*}
and since $X, X' \geq 1$ almost surely the first claim is immediate.

Similarly, if $x' \geq 1$ and $\ep \leq 1/6$, we have $\frac{\ep}{x'(x'-\ep)} \leq \frac 1 5$; hence almost surely
\begin{equation*}
\frac{1}{\ep^{-1} X - 2} - \frac{1}{\ep^{-1} X' - 1} \geq \frac{\ep}{X} - \frac{\ep}{X'} - \frac{\ep^2}{X'(X' - \ep)} \geq \ep\left(\frac 1 X - \frac{1}{X'} - \frac 1 5\right)\,,
\end{equation*}
and taking expectations yields the claim.
\end{proof}

\begin{lemma}\label{prior-first-moments-match}
Let $\mathrm{Q}$ and $\mathrm{Q}'$ be defined as in \eqref{eq:py} and \eqref{eq:py-prime}. If $Y \sim \mathrm{Q}$ and $Y' \sim \mathrm{Q'}$, then
\begin{equation*}
\E Y = \E Y' \leq 6\,.
\end{equation*}
\end{lemma}
\begin{proof}
It follows directly from the definition that
\begin{equation*}
\begin{split}
\E Y & = 1 + \frac 1{\cZ_\ep} \left(\int \frac{y}{(y-1)(y-2)} \mathrm{P_\ep}(\mathrm d y) - \Delta_\ep\right)\\
& = 1 + \frac 1{\cZ_\ep} \int \frac{1}{y-2} \mathrm{P_\ep}(\mathrm d y)
\end{split}
\end{equation*}
and analogously
\begin{equation*}
\begin{split}
\E Y' & = 2 + \frac 1{\cZ_\ep} \left(\int \frac{y'}{(y'-1)(y'-2)} \mathrm{P}'_{\ep}(\mathrm d y') - 2\Delta_{\ep'}\right)\\
& = 2 + \frac 1{\cZ_\ep} \int \frac{1}{y'-1} \mathrm{P}'_{\ep}(\mathrm d y')\,,
\end{split}
\end{equation*}
which implies
\begin{equation*}
\begin{split}
\E Y' - \E Y & = 1 - \frac 1{\cZ_\ep} \left(\int \frac{1}{y - 2} \dd \mathrm{P}_\ep(y) - \int \frac{1}{y' - 1} \dd \mathrm{P'}_\ep (y')\right) \\
& = 1 - \frac 1{\cZ_\ep} (\cZ_\ep) = 0\,.
\end{split}
\end{equation*}

To verify the inequality, note that
\begin{equation*}
\int \frac{1}{y-2} \mathrm{P_\ep}(\mathrm d y) = \E \frac{1}{\ep^{-1}X - 2} = \ep \E \frac{1}{X - 2\ep}\,,
\end{equation*}
and the fact that $X \geq 1$ almost surely and $\ep \leq 1/6$ implies that this quantity is bounded by $\frac 32 \ep$.
By \cref{rescaled-prior-bounds}, $\cZ_\ep \geq \frac{3}{10} \ep$; therefore
\begin{equation*}
\E Y = 1 + \frac 1{\cZ_\ep} \int \frac{1}{y-2} \mathrm{P_\ep}(\mathrm d y) \leq 1 + \frac{\frac 32 \ep}{\frac{3}{10} \ep} = 6\,.
\end{equation*}
\end{proof}

\begin{lemma}\label{approximate-probabilities}
If $Q = \frac 1m (U_1, \dots, U_m)$ and $Q' = \frac 1m (V_1, \dots, V_m)$, where $U$ and $V$ satisfy the assumptions of \cref{priors} with $\ep = \frac 1 {24} \delta$, and recall that $E = \{Q \in \tilde \cD_{m, \delta}^-\}$ and $E' = \{Q \in \tilde \cD_{m}^+\}$. Then
\begin{equation*}
\p[E^C]+ \p[{E'}^C] \lesssim \frac{L^4}{\delta^2 m}\,.
\end{equation*}
\end{lemma}
\begin{proof}
We first show that $E$ holds as long as $\frac 1 m \sum_{i=1}^m U_i^2 \leq 8$ and $\frac 1 m \sum_{i=1}^m |U_i - 1| \leq \delta$.
Indeed, if we set $s \defeq \sum_{i=1}^m Q_i$, where $Q_i=U_i/m$, then the second condition together with Jensen's inequality imply that $|s - 1| \leq \delta$.
Therefore, recalling that $\bar Q=Q/s$, we get
\begin{equation*}
\chi^2(\bar Q, u) + 1 = m \sum_{i=1}^m {\bar Q_i}^2 = \frac{1}{s^2m} \sum_{i=1}^m U_i^2 \leq \frac{8}{s^2} \leq \frac{8}{1-2 \delta} \leq 10\,,
\end{equation*}
where the final inequality follows from the assumption that $\delta \leq 1/10$.
Hence $\chi^2(\bar Q, u) \leq 9$.

Likewise,
\begin{equation*}
\tv{\bar Q}{u} = \frac 1 2 \sum_{i=1}^m \left|\bar{Q_i} - \frac 1 m\right| \leq \frac 12 \sum_{i=1}^m \left|\bar{Q_i} - Q_i\right| + \frac 1 2 \sum_{i=1}^m \left|Q_i - \frac 1 m \right| \leq \frac 12 (|s-1|+ \delta) \leq \delta\,.
\end{equation*}

By Chebyshev's inequality combined with the assumption that $\E U^2 \leq 6$,
\begin{equation*}
\p\left[\frac 1 m \sum_{i=1}^m U_i^2 > 8 \right] \leq \frac{\var(U^2)}{4 m}\,,
\end{equation*}
and since $U \leq 16 \ep^{-1} L^2$ almost surely,
\begin{equation*}
\var(U^2) \leq \E U^4 \leq (16 \ep^{-1} L^2)^2 \E U^2 \lesssim \delta^{-2} L^4\,.
\end{equation*}
Moreover, since $\ep = \frac 1{24} \delta$ and $\E |U - 1| \leq 12 \ep$,
\begin{equation*}
\p\left[\frac 1m \sum_{i=1}^m |U_i - 1| > \delta\right] \leq \frac{4 \var(|U-1|)}{\delta^2 m} \lesssim \frac{1}{\delta^2 m}\,.
\end{equation*}
Combining these bounds yields
\begin{equation*}
\p[E^C] \lesssim \frac{L^4}{\delta^2 m}\,.
\end{equation*}

The argument for $E'$ is analogous.
It suffices for $E'$ to hold that $\frac 1m \sum_{i=1}^m V_i^2 \leq 8$ along with the conditions $\frac 1 m \sum_{i=1}^m |V_i - 1| \geq \frac 3 4$ and $\left|(\frac 1m \sum_{i=1}^m V_i) - 1\right| \leq \delta$.
As above, the first condition is violated with probability at most $C \frac{L^4}{\delta^2 m}$, and the latter two conditions are violated with probability at most $C \frac{1}{\delta^2 m}$.
The claim follows.
\end{proof}

\subsection{Omitted proofs}\label{subsec:omitted}

\begin{proof}[Proof of \cref{prop:almost-spike}]
Denote by $\rho^{(1)}$ and $\rho^{(2)}$ the distributions of $Z^{(1)}$ and $Z^{(2)}$, respectively.
It suffices to show $\wpk(\nuone, \nutwo) + W_p(\rho^{(1)}, \rho^{(2)}) \geq W_p(\nuone, \nutwo)$.
Let $U \in \stie$ be such that the column span of $U$ agrees with the subspace $\cU$ on which the random variables $X^{(1)}$ and $X^{(2)}$ are supported in the definition of $\muone$ and $\mutwo$.
We let $(Y^{(1)}, Y^{(2)})$ be a pair of random variables such that marginally $Y^{(i)} \sim \mui_U$ for $i \in \{1, 2\}$ and
\begin{equation*}
W^p_p(\muone_U, \mutwo_U) = \E \|Y^{(1)} - Y^{(2)}\|^p\,.
\end{equation*}
Likewise, let $(W^{(1)}, W^{(2)})$ be a coupling of $\rho^{(1)}$ and $\rho^{(2)}$ such that
\begin{equation*}
W^p_p(\rho^{(1)}, \rho^{(2)}) = \E \|W^{(1)} - W^{(2)}\|^p\,.
\end{equation*}
Then~\eqref{eq:almost-spike} implies that $(Y^{(1)} + W^{(1)}, Y^{(2)} + W^{(2)})$ is a coupling of $\muone$ and $\mutwo$. Therefore
\begin{align*}
W_p(\muone, \mutwo) & \leq \left(\E \|Y^{(1)} + W^{(1)} - (Y^{(2)} + W^{(2)})\|^p\right)^{1/p} \\
& \leq \left(\E \|Y^{(1)} - Y^{(2)}\|^p\right)^{1/p} + \left(\E \|W^{(1)} - W^{(2)}\|^p\right)^{1/p} \\
& = W_p(\muone_U, \mutwo_U) + W_p(\rho^{(1)}, \rho^{(2)})\\
& \leq \wpk(\muone, \mutwo) + W_p(\rho^{(1)}, \rho^{(2)})\,.
\end{align*}
\end{proof}

\begin{proof}[Proof of \cref{lem:t1_is_subg}]
This can be deduced from a series of well known facts. By \cref{prop:bobgot}, it suffices to show that for $f$ Lipschitz,
\begin{equation*}
\E e^{\lambda (f(X) - \E f(X))} \leq e^{C^2 k \lambda^2 \sigma^2/2}\,.
\end{equation*}
By symmetrization, we can bound
\begin{equation*}
\E e^{\lambda (f(X) - \E f(X))} \leq \E e^{\lambda \ep(f(X) - f(X'))} \leq \E e^{\lambda \ep \|X - X'\|}\,,
\end{equation*}
where $X' \sim \mu$ is independent of $X$. Continuing, we obtain
\begin{equation*}
\E e^{\lambda \ep \|X - X'\|} \leq \E e^{2 \lambda \ep \|X\|} \leq e^{16 k \lambda^2 \sigma^2}
\end{equation*}
by \cref{prop:norm_subG}.
This proves the claim with $C = 32$.
\end{proof}

\begin{proof}[Proof of \cref{thm:concentration_is_tp}]
First, assume that $\mu$ satisfies the $\tp{\sigma^2}$ inequality.
Then, by \citet[Proposition 1.9]{GozLeo10}, the product measure $\mu^{\otimes n}$ satisfies the inequality $\tp{n^{2/p-1}\sigma^2 }$ on the space $\cX^n$ equipped with the metric $\rho_p(x, y) \defeq \left(\sum_{i=1}^n \rho(x_i, y_i)^p \right)^{1/p}$.
Therefore, $\mu^{\otimes n}$ also satisfies $\tone{n^{2/p - 1}\sigma^2 }$ with respect to this same metric.

We now note that the function
\begin{equation*}
(x_1, \dots x_n) \mapsto W_p\left(\frac 1n \sum_{i=1}^n \delta_{x_i}, \mu \right)
\end{equation*}
is $n^{-1/p}$-Lipschitz with respect to $\rho_p$.
Indeed, we have
\begin{align*}
\Big|W_p\Big(\frac 1n \sum_{i=1}^n \delta_{x_i}, \mu \Big) - W_p\Big(\frac 1n \sum_{i=1}^n \delta_{y_i}, \mu \Big)\Big| & \leq W_p\Big(\frac 1n \sum_{i=1}^n \delta_{x_i},\frac 1n \sum_{i=1}^n \delta_{y_i}\Big) \\
& \leq \Big(\frac 1n \sum_{i=1}^n \rho(x_i, y_i)^p \Big)^{1/p}\\
& = n^{-1/p} \rho_p(x, y)\,.
\end{align*}

Combining these observations with \cref{prop:bobgot}, we obtain
\begin{equation*}
\E e^{\lambda(W_p(\mu_n, \mu) - \E W_p(\mu_n, \mu))} \leq e^{\lambda^2 n^{-2/p} n^{2/p-1} \sigma^2/2} = e^{\lambda^2 n^{-1} \sigma^2/2} \qquad \forall \lambda \in \RR\,.
\end{equation*}

In the other direction, suppose that $W_p(\mu_n, \mu)$ is $\sigma^2/n$-subgaussian.
A Chernoff bound implies the concentration inequality
\begin{equation*}
\p[W_p(\mu_n, \mu) - \E W_p(\mu_n, \mu) \geq t] \leq e^{nt^2/2\sigma^2}\,.
\end{equation*}
We follow the proof of~\citet[Theorem 3.4]{Goz09}, and only sketch the argument here.
It is easy to verify that the function $W_p(\cdot, \mu)$ is lower semi-continuous with respect to the weak topology~\citep[see][proof of Theorem 4.1]{Vil09}, which implies that, for any $t \geq 0$, the set $\cO_t \defeq \{\nu \in \cP(\cX) : W_p(\nu, \mu) > t\}$ is open.
Moreover, as $\mu \in \cP_p$, \citet[Theorem 6.9]{Vil09} combined with Varadarajan's theorem yields that $\E W_p(\mu_n, \mu) \to 0$ as $n \to \infty$~\citep[see][proof of Theorem 22.22]{Vil09}.

Applying Sanov's theorem then yields
\begin{align*}
- \inf_{\nu \in \cO_t} D(\nu \| \mu) & \leq \liminf_{n \to \infty} \frac 1n \log \p[W_p(\mu_n, \mu) > t]  \\
& \leq \liminf_{n \to \infty} \frac 1n \cdot \frac{n \max\{(t- \E W_p(\mu_n, \mu))^2, 0\}}{2 \sigma^2} \\
& = \frac{t^2}{2 \sigma^2}\,.
\end{align*}
We obtain that, for any $t \geq 0$,
\begin{equation*}
D(\nu \| \mu) \geq - \frac{t^2}{2 \sigma^2} \qquad \forall \nu \in \cP(\cX) \text{ s.t. } W_p(\nu, \mu) > t\,.
\end{equation*}
Therefore $\mu$ satisfies $\tp{\sigma^2}$.
\end{proof}

\begin{proof}[Proof of \cref{thm:p-pprime-concentration}]
As in the proof of \cref{thm:concentration_is_tp}, we have that $\mu^{\otimes n}$ satisfies $\tone{n^{2/p'-1} \sigma^2}$ with respect to $\rho_{p'}(x, y)$. We also that that
\begin{equation*}
(x_1, \dots x_n) \mapsto W_p\left(\frac 1n \sum_{i=1}^n \delta_{x_i}, \mu \right)
\end{equation*}
is $n^{-1/p}$-Lipschitz with respect to $\rho_{p'}$, since the fact that $p' \leq p$ implies
\begin{equation*}
\Big|W_p\Big(\frac 1n \sum_{i=1}^n \delta_{x_i}, \mu \Big) - W_p\Big(\frac 1n \sum_{i=1}^n \delta_{y_i}, \mu \Big)\Big| \leq n^{-1/p} \rho_p(x, y) \leq n^{-1/p} \rho_{p'}(x, y)\,.
\end{equation*}
We conclude as in the first part of the proof of \cref{thm:concentration_is_tp}.
\end{proof}

\begin{proof}[Proof of \cref{thm:sin}]
We show that for any measures $\nuone$, $\nutwo$, if $\cV = \mathrm{span}(\hat V)$ where $V = \argmax_{U \in \stie} W_p(\nuone_U, \nutwo_U)$, then
\begin{equation*}
\sin^2\big(\measuredangle(\cV, \cU)\big) \lesssim \frac{\wpk(\muone, \nuone) + \wpk(\mutwo, \nutwo)}{W_p(\muone, \mutwo)}\,.
\end{equation*}
Combined with \cref{prop:supremum}, this implies the claim.

\Cref{prop:low-dim-coupling} guarantees the existence of an optimal coupling between $\muone$ and $\mutwo$ such that if $(X, Y) \sim \gamma$, then $X - Y$ lies in $\cU$ almost surely.
Let $U \in \stie$ be such that the column span of $U$ is $\cU$.
We have
\begin{align*}
W_p(\muone_{V}, \mutwo_{V}) & \leq (\E \|V (X - Y) \|^p)^{1/p} \\
& = (\E \|V U^\top U (X - Y) \|^p)^{1/p} \\
& \leq \op{V U^\top} (\E \|X - Y\|^p)^{1/p} = \op{V U^\top} W_p(\muone, \mutwo)\,.
\end{align*}

By the definition of $V$,
\begin{align*}
W_p(\muone, \mutwo) - W_p(\muone_V, \mutwo_V) &  = W_p(\muone_U, \mutwo_U) - W_p(\muone_V, \mutwo_V) \\
& \leq W_p(\muone_U, \mutwo_U) - W_p(\nuone_U, \nutwo_U) \\
& \quad \quad \quad + W_p(\nuone_V, \nutwo_V) - W_p(\muone_V, \mutwo_V) \\
& \le 2 \sup_{U \in \stie} |W_p(\muone_U, \mutwo_U) - W_p(\nuone_U, \nutwo_U)| \\
& \leq 2\wpk(\muone, \nuone) + 2\wpk(\mutwo, \nutwo)\,.
\end{align*}
Therefore
\begin{equation*}
(1 - \op{V U^\top}) W_p(\muone, \mutwo) \le 2 \wpk(\muone, \nuone) + 2\wpk(\mutwo, \nutwo)\,.
\end{equation*}

Moreover,
\begin{align*}
\sin^2 \big(\measuredangle(\cV, \cU)\big) & = 1 - \sup_{\substack{v \in \cV, u \in \cU \\ \|u\| = \|v\| = 1}}(u^\top v)^2 \\
& = 1 - \sup_{\substack{x, y \in \RR^d \\ \|x\| = \|y\| = 1}} ((Ux)^\top (Vy))^2 \\
& = 1 - \op{V U^\top}^2 \\
& \leq 2(1 - \op{V U^\top})\,,
\end{align*}
since $\op{V U^\top} \leq 1$.
The claim follows.
\end{proof}
\section*{Acknowledgements}
The authors thank Oded Regev for helpful suggestions.
\bibliographystyle{abbrvnat_weed}
\bibliography{RigWee19ext}
\end{document}